\documentclass[final]{siamltex}
\usepackage{amsfonts,amsmath,amssymb}
\usepackage{mathrsfs,mathtools}
\usepackage{enumerate}
\usepackage{xspace,MACROS/mydef} 
\usepackage{hyperref}
\usepackage{esint}
\usepackage{graphicx}
\usepackage{placeins}
\DeclareMathAlphabet{\mathpzc}{OT1}{pzc}{m}{it}

\usepackage[notcite,notref]{showkeys} % To see crossreferences.
\usepackage[usenames,dvipsnames]{color}
\newcommand{\EO}[1]{#1}
\newtheorem{remark}[theorem]{Remark}
\numberwithin{equation}{section}

\newcommand{\abs}[1]{\lvert#1\rvert}

\DeclareMathOperator{\proj}{proj}

\DeclareMathOperator*{\esssup}{esssup}

\title{Error estimates for the optimal control of a parabolic fractional PDE}

\author{Christian Glusa\thanks{Center for Computing Research, Sandia National Laboratories, Albuquerque, NM 87123, USA ({\tt caglusa@sandia.gov}).}
\and
Enrique Ot\'arola\thanks{Departamento de Matem\'atica, Universidad T\'ecnica Federico Santa Mar\'ia, Valpara\'iso, Chile (\texttt{enrique.otarola@usm.cl}}).}

\date{Draft version of \today.}

\begin{document}

\maketitle
\begin{abstract}
We consider the integral definition of the fractional Laplacian and analyze a linear-quadratic optimal control problem for the so-called fractional heat equation; control constraints are also considered. We derive existence and uniqueness results, first order optimality conditions, and regularity estimates for the optimal variables. To discretize the state equation equation we propose a fully discrete scheme that relies on an implicit finite difference discretization in time combined with a piecewise linear finite element discretization in space. We derive stability results and a novel $L^2(0,T;L^2(\Omega))$ a priori error estimate. On the basis of the aforementioned solution technique, we propose a fully discrete scheme for our optimal control problem that discretizes the control variable with piecewise constant functions and derive a priori error estimates for it. We illustrate the theory with one- and two-dimensional numerical experiments.
\end{abstract}

%
%%%%%%%%%%%%%%%%%%%%%%%%%%%%%%%%%%%%%%%%%%%%%%%%%%%%%%%%%%%%%%%%%%%%%%%%%%%%%%%%%%%%%%
\section{Introduction}
\label{sec:introduccion}
%%%%%%%%%%%%%%%%%%%%%%%%%%%%%%%%%%%%%%%%%%%%%%%%%%%%%%%%%%%%%%%%%%%%%%%%%%%%%%%%%%%%%%
In this work we shall be interested in the design and analysis of solution techniques for a linear-quadratic optimal control problem involving an initial boundary value problem for a fractional parabolic equation. To make matters precise, for $n \geq 1$, we let $\Omega \subset \mathbb{R}^n$ be an open and bounded domain with Lipschitz boundary $\partial\Omega$; \EO{when deriving regularity and error estimates we will assume that $\Omega$ is smooth.} Given a desired state $\usf_d: \Omega \times (0,T) \rightarrow \mathbb{R}$ and a regularization parameter $\mu>0$, we define the cost functional
\begin{equation}	
\label{def:J}
 J(\usf,\zsf)= \frac{1}{2}\int_{0}^T \left( \| \usf - \usf_{d} \|^2_{L^2(\Omega)} + \mu \| \zsf\|^2_{L^2(\Omega)} \right) \diff t. 
\end{equation}
Let $\fsf:\Omega \times (0,T) \rightarrow \mathbb{R}$ and $\usf_0:\Omega \to \mathbb{R}$ be fixed functions. We will call them the right-hand side and initial datum, respectively. Let $s \in (0,1)$. We shall be concerned with the following PDE-constrained optimization problem: Find 
\begin{equation}
\label{eq:min}
  \min J(\usf,\zsf)
\end{equation}
subject to the \emph{fractional heat equation}
\begin{equation}
\label{eq:fractional_heat}
\partial_t \usf + (-\Delta)^s \usf = \fsf + \zsf  \text{ in } \Omega \times (0,T), 
\quad \usf = 0 \text{ in } \Omega^c \times (0,T),
\quad \usf(0) = \usf_0 \text{ in } \Omega,
\end{equation}
and the \emph{control constraints}
\begin{equation}
\label{eq:cc}
\asf(x,t) \leq \zsf(x,t) \leq \bsf(x,t) \quad\textrm{a.e.~~} (x,t) \in Q:= \Omega \times (0,T).
\end{equation}
The functions $\asf$ and $\bsf$ both belong to $L^2(Q)$ and satisfy the property $\asf(x,t) \leq \bsf(x,t)$ for almost every $(x,t) \in Q$. In \eqref{eq:fractional_heat}, $\Omega^c := \mathbb{R}^n \setminus \Omega$. For convenience, we will refer to the optimal control problem \eqref{eq:min}--\eqref{eq:cc} as the \emph{parabolic fractional optimal control problem}; see section~\ref{sec:control} for its precise description and analysis.

\EO{For smooth functions $w: \mathbb{R}^n \rightarrow \mathbb{R}$, there are several equivalent definitions of the fractional Laplace operator $(-\Delta)^s$ in $\mathbb{R}^n$ \cite{MR3613319}. Indeed, $(-\Delta)^s$ can be naturally defined via Fourier transform:
\begin{equation}
\label{eq:Fourier}
\mathcal{F}( (-\Delta)^s w) (\xi) = | \xi |^{2s} \mathcal{F}(w) (\xi).
\end{equation}
Equivalently, $(-\Delta)^s$ can be defined by means of the following pointwise formula:
\begin{equation}
 (-\Delta)^s w(x) = C(n,s) \, \mathrm{ p.v } \int_{\mathbb{R}^n} \frac{w(x) - w(y)}{|x-y|^{n+2s}} \mathrm{d}y,
 \qquad
 C(n,s) = \frac{2^{2s} s \Gamma(s+\frac{n}{2})}{\pi^{n/2}\Gamma(1-s)},
 \label{eq:integral_definition}
\end{equation}
where $\textrm{p.v}$ stands for the Cauchy principal value and $C(n,s)$ is a positive normalization constant that depends only on $n$ and $s$ \cite[equation (3.2)]{NPV:12}; $C(n,s)$ is introduced to guarantee that the symbol of the resulting operator is $| \xi |^{2s}$. A proof of the equivalence of these two definitions can be found in \cite[section 1.1]{Landkof} and \cite[Proposition 3.3]{NPV:12}. In addition to \eqref{eq:Fourier} and \eqref{eq:integral_definition}, several other equivalent definitions of $(-\Delta)^s$ in $\mathbb{R}^n$ are available in the literature \cite{MR3613319}. For instance, the ones based on the Balakrishnan formula and a suitable harmonic extension \cite{CS:07}. In bounded domains there are also several definitions of $(-\Delta)^s$. 
%available in the literature; see the discussion in \cite[Section 2]{MR3393253} and \cite[Section 6]{MR3503820}. 
For functions supported in $\bar \Omega$, we may utilize the integral representation \eqref{eq:integral_definition} to define $(-\Delta)^s$. This gives rise to the so-called \emph{restricted} or \emph{integral} fractional laplacian. Notice that we have materialized a zero Dirichlet condition by restricting the operator to act only on functions that are zero outside $\Omega$. We must immediately mention that in bounded domains, and in addition to the \emph{restricted} or \emph{integral} fractional Laplacian there are, at least, two others \emph{non-equivalent} definitions of nonlocal operators related to the fractional Laplacian: the \emph{regional} fractional Laplacian and the \emph{spectral} fractional Laplacian; see the discussion in \cite[Section 2]{MR3393253} and \cite[Section 6]{MR3503820}. We adopt the restricted or integral definition of the fractional Laplace operator $(-\Delta)^s$, which, from now on, we shall simply refer to as the \emph{integral fractional Laplacian}.}

Since the seminal work of Caffarelli and Silvestre \cite{CS:07}, the analysis of regularity properties of solutions to fractional partial differential equations (PDEs) has received a tremendous attention: fractional diffusion has been one of the most studied topics in the past decade \cite{CS:07,MR3276603,MR3504596,MR3348172,MR3168912,MR2270163}. Such an analysis has been motivated, in part, by the fact that the integral fractional Laplacian of order $2s$ corresponds to the infinitesimal generator of a $2s$-stable L\'evy process. These processes have been widely employed for modeling market fluctuations for both risk management and option pricing purposes \cite{MR2042661}. Further applications of fractional diffusion include material science (e.g. subsurface flow where nonlocal porous media models accurately describe the physical process) \cite{BensonWheatcraftEtAl2000_ApplicationFractionalAdvectionDispersionEquation,Biler2015,Silling2000_ReformulationElasticityTheoryDiscontinuities}, nonlocal electrostatics \cite{doi:10.1063/1.2819487}, image processing \cite{GilboaOsher2008_NonlocalOperatorsWithApplications,MR2578033}, fluids \cite{doi:10.1063/1.2208452}, predator search behaviour \cite{Sims}, and many others. It is then only natural that interest in efficient approximation schemes for these problems arises and that one might be interested in their control.

The study of solution techniques for problems involving fractional diffusion is a relatively new but rapidly growing area of research and thus it is impossible to provide a complete overview of the available results and limitations. We restrict ourselves to referring the interested reader to \cite{MR3893441} for a survey. In contrast to these advances, the study of solution techniques for PDE-constrained optimization problems involving fractional and nonlocal equations have not been fully developed. To the best of our knowledge, one of the first works in the elliptic setting is \cite{MR3158780}, where the authors consider an optimal control problem for a general nonlocal diffusion operator with finite range interactions. Later, an \emph{elliptic} optimal control problem for the \emph{spectral} fractional powers of elliptic operators was analyzed in \cite{AO}; numerical \EO{schemes} were also proposed and studied. Recently, a similar PDE-constrained optimization problem, but for the \emph{integral} fractional Laplacian, has been considered in \cite{DGO}. In this work, the authors analyze the underlying control problem, derive regularity estimates, propose numerical schemes, and derive a priori error estimates. We also mention \cite{semilinear}, where an optimal control problem for a fractional semilinear equation is considered.
Concerning parabolic optimal control problems, the first work that propose and study numerical schemes when the state equation is the fractional heat equation is \cite{MR3504977}. In this work, the authors consider the \emph{spectral} fractional powers of elliptic operators and derive error estimates for a fully discrete scheme that approximates the solution of the underlying optimal control problem. To close this paragraph, we would like to stress that the \EO{\emph{integral} and \emph{spectral}} definitions of the fractional Laplace operator \emph{do not coincide}. \EO{This, in particular, implies that the boundary behavior of solutions to
\begin{equation}
\label{eq:Dir_frac}
(-\Delta)^s \mathfrak{u} = \mathfrak{f} \textrm{ in } \Omega, 
\end{equation}
supplemented with suitable boundary conditions, is quite different depending on what definition for $(-\Delta)^s$ is adopted: integral or spectral. When the \emph{spectral} definition is considered, we supplement $-\Delta$ with homogeneous Dirichlet boundary conditions; $(-\Delta)^s$ is defined on the basis of eigenfunctions of $-\Delta$ that vanish on $\partial \Omega$. This gives rise to a suitable Dirichlet condition on problem \eqref{eq:Dir_frac}. In contrast, when the integral definition is considered, we supplement problem \eqref{eq:Dir_frac} with the Dirichlet condition $\mathfrak{u} = 0$ in $\Omega^c$. If $(-\Delta)^s$ corresponds to  \emph{integral} fractional Laplacian, $\Omega$ is smooth, and $\mathfrak{f} \in H^{1/2-s}(\Omega)$, then the solution $\mathfrak{u}$ of \eqref{eq:Dir_frac} is of the form \cite[formulas (7.7)--(7.12)]{MR3276603},}
\begin{equation}\label{boundary-grubb}
\mathfrak{u}(x) \approx \textrm{dist} (x,\partial\Omega)^s + \mathfrak{v} (x),
\end{equation}
with $\mathfrak{v}$ is smooth; hereafter $\textrm{dist} (x,\partial\Omega)$
indicates the distance from $x\in\Omega$ to $\partial\Omega$. \EO{In contrast, \cite[Theorem 1.3]{MR3489634} states that solutions of \eqref{eq:Dir_frac} with $\Laps$ being the \emph{spectral} fractional Laplacian, and $\Omega$ and $\mathfrak{f}$ being sufficiently smooth, behave as}
\begin{equation}\label{boundary-CS}
\begin{aligned}
\mathfrak{u}(x) \approx \textrm{dist} (x,\partial\Omega)^{2s} + \mathfrak{v}(x),
\quad
\mathfrak{u}(x) \approx \textrm{dist} (x,\partial\Omega)+ \mathfrak{v}(x)
\end{aligned}
\end{equation}
for $0 <s <\frac12$ and $\frac12 < s < 1$, respectively. The case $s = \frac12$ is exceptional; for $\Omega \subset \mathbb{R}^2$ with 
$\partial \Omega$ smooth, it holds that
$
\mathfrak{u}(x) \approx \textrm{dist} (x,\partial\Omega) \, |\log(\textrm{dist} (x,\partial\Omega))|+ \mathfrak{v}(x),
$
with $\mathfrak{v}$ smooth \cite{MR1204855}. This lack of boundary regularity is responsible for reduced rates of convergence when numerical schemes to approximate solutions are considered.

This exposition is the first one that studies approximations techniques for  \eqref{eq:min}--\eqref{eq:cc}. Let us briefly detail some of the main contributions of our work:
\begin{itemize}
\item[$\bullet$] \emph{Fractional heat equation}: We analyze the fractional heat equation \eqref{eq:fractional_heat}; we adopt the integral definition for $(-\Delta)^s$. We propose a fully discrete scheme to solve \eqref{eq:fractional_heat} and derive, in section \ref{sec:fully_scheme}, stability and error estimates for it. In particular, we obtain a novel $L^2(Q)$ a priori error estimate. 
%To the best of our knowledge, these results are not available in the literature.
\item[$\bullet$] \emph{Parabolic fractional optimal control problem}: We analyze the optimal control problem \eqref{eq:min}--\eqref{eq:cc}. We derive existence and uniqueness results together with first order necessary and sufficient optimality conditions. In addition, we derive regularity estimates for the optimal variables.
%; the latter being important since, as it is well known, smoothness and rate of convergence for numerical methods go hand in hand. 
Notice that in view of \eqref{boundary-CS} and \eqref{boundary-grubb}, the derived regularity estimates are in sharp contrast with the ones \EO{obtained} in \cite[Theorem 17]{MR3504977}.
\item[$\bullet$] \emph{Fully discrete approximation:} We propose an implicit fully discrete approximation for the optimal control problem \eqref{eq:min}--\eqref{eq:cc}. We derive first order optimal conditions and perform an a priori error analysis. 
\item[$\bullet$] \emph{Complexity:} The approach taken to discretize and solve \eqref{eq:fractional_heat} is quite different from the spectral case discussed in \cite{MR3504977}. While \cite{MR3504977} transforms \eqref{eq:fractional_heat} into a quasi--stationary elliptic problem with a dynamic boundary condition on a \(n+1\)-dimensional domain, we directly discretize the integral form using matrix compression techniques to obtain quasi--optimal complexity. 
\end{itemize}

The outline of this paper is as follows. The notation and functional setting is described in section \ref{sec:Prelim}. In section \ref{sub:stateequation}, we derive the existence and uniqueness of a weak solution for problem \eqref{eq:fractional_heat}. In addition, we present energy estimates and review regularity results. In section \ref{sec:control}, we study the \emph{parabolic fractional optimal control problem} \EO{and} derive regularity estimates for the optimal variables. In section \ref{sec:a_priori_state}, we introduce a fully discrete scheme for \eqref{eq:fractional_heat}: we consider the standard backward Euler scheme for time discretization and a piecewise linear finite element discretization in space. For $s \in (0,1)$, we derive discrete stability results and \EO{an a priori error estimate in $L^2(Q)$}. Section \ref{sec:approximation_control} is devoted to the design and analysis of a numerical scheme to approximate the control problem \eqref{eq:min}--\eqref{eq:cc}. In particular, in section \ref{sub:apriori_control}, we derive a priori error estimates. Finally, section \ref{sec:numerical-examples} presents one- and two-dimensional numerical experiments that illustrate the theory developed in section \ref{sub:apriori_control}.

%%%%%%%%%%%%%%%%%%%%%%%%%%%%%%%%%%%%%%%%%%%%%%%%%%%%%%%%%%%%%%%%%%%%%%%%%%%%%%%%%%%%%%
\section{Notation and preliminaries}
\label{sec:Prelim}
%%%%%%%%%%%%%%%%%%%%%%%%%%%%%%%%%%%%%%%%%%%%%%%%%%%%%%%%%%%%%%%%%%%%%%%%%%%%%%%%%%%%%%
In this section, we will introduce some notation and the set of assumptions that we shall operate under.
%%%%%%%%%%%%%%%%%%%%%%%%%%%%%%%%%%%%%%%%%%%%%%%%%%%%%%%%%%%%%%%%%%%%%%%%%%%%%%%%%%%%%%
\subsection{Notation}
\label{sub:notation}
%%%%%%%%%%%%%%%%%%%%%%%%%%%%%%%%%%%%%%%%%%%%%%%%%%%%%%%%%%%%%%%%%%%%%%%%%%%%%%%%%%%%%%

Throughout this work $\Omega$ is an open and bounded domain with Lipschitz boundary $\partial \Omega$; \EO{when deriving regularity and error estimates we will assume that $\Omega$ is smooth.} The complement of $\Omega$ will be denoted by $\Omega^c$. If $T >0$ is a fixed time, we set $Q = \Omega \times (0,T)$. Whenever $\Xcal$ is a normed space we denote by $\| \cdot \|_{\Xcal}$ its norm and by $\Xcal'$ its dual. For normed spaces $\Xcal$ and $\Ycal$, we write $\Xcal \hookrightarrow \Ycal$ to indicate that $\Xcal$ is continuously embedded in $\Ycal$.

If $D\subset \R^{n}$ is open and $\phi: D \times [0,T] \to \R$, we consider $\phi$ as a function of $t$ with values in a Banach space $\Xcal$, \ie
$
 \phi:[0,T] \ni t \mapsto  \phi(t) \equiv \phi(\cdot,t) \in \Xcal.
$
For $1 \leq p \leq \infty$, $L^p( 0,T; \Xcal )$ is the space of $\Xcal$-valued functions whose $\Xcal$-norm is in $L^p(0,T)$. This is a Banach space for the norm
\[
  \| \phi \|_{L^p( 0,T;\Xcal)} = \left( \int_0^T \| \phi(t) \|^p_\Xcal \diff t\right)^{\hspace{-0.1cm}\frac{1}{p}} 
  , \quad 1 \leq p < \infty, \quad
  \| \phi \|_{L^\infty( 0,T;\Xcal)} = \esssup_{t \in (0,T)} \| \phi(t) \|_\Xcal.
\]

The relation $a \lesssim b$ indicates that $a \leq Cb$ with a nonessential constant $C$ that might change at each occurrence. 

%%%%%%%%%%%%%%%%%%%%%%%%%%%%%%%%%%%%%%%%%%%%%%%%%%%%%%%%%%%%%%%%%%%%%%%%%%%%%%%%%%%%%%
\subsection{Function spaces}
\label{sub:function_spaces}
%%%%%%%%%%%%%%%%%%%%%%%%%%%%%%%%%%%%%%%%%%%%%%%%%%%%%%%%%%%%%%%%%%%%%%%%%%%%%%%%%%%%%%
For any $s \geq 0$, we define $H^s(\mathbb{R}^n)$, the Sobolev space of order $s$ over $\mathbb{R}^n$, by \cite[Definition 15.7]{Tartar}
\[
 H^s(\mathbb{R}^n) := \left \{ v \in L^2(\mathbb{R}^n): (1+|\xi|^2)^{s/2} \mathcal{F}(v) \in L^2(\mathbb{R}^n)\right \}.
\]
With the space $H^s(\mathbb{R}^n)$ at hand, we define $\tilde H^s(\Omega)$ as the closure of $C_0^{\infty}(\Omega)$ in $H^s(\mathbb{R}^n)$. This space can be equivalently characterized by \cite[Theorem 3.29]{McLean}
\begin{equation}
\tilde H^s(\Omega) = \{v|_{\Omega}: v \in H^s(\mathbb{R}^n), \textrm{ supp } v \subset \overline\Omega\}.
\end{equation}
When $\partial \Omega$ is Lipschitz, $\tilde H^s(\Omega)$ is equivalent to $\mathbb{H}^s(\Omega)=[L^2(\Omega),H_0^1(\Omega)]_s$, the real interpolation between $L^2(\Omega)$ and $H_0^1(\Omega)$, for $s \in (0,1)$ and to $H^s(\Omega) \cap H_0^1(\Omega)$ for $s \in (1,3/2)$ \cite[Theorem 3.33]{McLean}. We denote by $H^{-s}(\Omega)$ the dual space of $\tilde H^s(\Omega)$ and by $\langle \cdot, \cdot \rangle$ the duality pair between these two spaces. We also define the bilinear form
\begin{equation}
\label{eq:bilinear_form}
 \mathcal{A}(v,w) = \frac{C(n,s)}{2} \iint_{\mathbb{R}^n \times \mathbb{R}^n} \frac{( v(x) - v(y) ) (w(x)-w(y))}{|x-y|^{n+2s}} \mathrm{d}x \mathrm{d}y.
\end{equation}
We denote by $\| \cdot \|_s$ the norm that $ \mathcal{A}(\cdot,\cdot)$ induces;
%which is just 
a multiple of the $H^s(\mathbb{R}^n)$-seminorm:
\EO{$
 \| v \|_s =  \sqrt{\mathcal{A}(v,v)}= \sqrt{\mathfrak{C}(n,s)} |v|_{H^s(\mathbb{R}^n)},
$
where $\mathfrak{C}(n,s) = \sqrt{C(n,s) / 2}$.}

%%%%%%%%%%%%%%%%%%%%%%%%%%%%%%%%%%%%%%%%%%%%%%%%%%%%%%%%%%%%%%%%%%%%%%%%%%%%%%%%%%%%%%
\subsection{Elliptic regularity}
\label{sub:elliptic_regularity}
%%%%%%%%%%%%%%%%%%%%%%%%%%%%%%%%%%%%%%%%%%%%%%%%%%%%%%%%%%%%%%%%%%%%%%%%%%%%%%%%%%%%%%
Let $f \in H^{-s}(\Omega)$. Since the bilinear form $\mathcal{A}$ is continuous and coercive, an application of the Lax-Milgram Lemma immediately yields the well-posedness of the following elliptic problem: Find $u \in \tilde H^s(\Omega)$ such that
\begin{equation}
\label{eq:elliptic_problem}
\mathcal{A}(u,v) = \langle f , v \rangle \quad \forall v \in \tilde H^s(\Omega).
\end{equation}

When $\partial \Omega$ is smooth the following regularity properties for $u$ can be derived.

\begin{proposition}[Sobolev regularity of $u$ on smooth domains]
Let $s \in (0,1)$ and $\Omega$ be a domain such that $\partial \Omega \in C^{\infty}$. If $f \in H^{r}(\Omega)$, for some $r \geq -s$, then the solution $u$ of problem \eqref{eq:elliptic_problem} belongs to $H^{s + \vartheta}(\Omega)$, where $\vartheta = \min \{s+r,1/2-\epsilon \}$ and $\epsilon > 0$ is arbitrarily small. In addition, the following estimate holds:
\begin{equation}
\| u \|_{H^{s+\vartheta}(\Omega)} \lesssim \| f \|_{H^{r}(\Omega)},
\label{eq:regularity_state_smooth}
 \end{equation}
where the hidden constant depends on $\Omega$, $n$, $s$, and $\vartheta$.
\label{pro:state_regularity_smooth}
\end{proposition}
%%%
\begin{proof}
\EO{See \cite{MR0185273,MR3276603}.}
\end{proof}

As a consequence of the previous result,  it can be observed that smoothness of $f$ does not ensure that solutions are any smoother than \EO{$\cap \{ H^{s+1/2-\epsilon}(\Omega): \epsilon > 0 \}$.}

When $\Omega$ is a bounded Lipschitz domain satisfying the exterior ball condition, the following regularity estimate can be derived \cite{MR3168912}: If $f \in L^{\infty}(\Omega)$, then $u \in C^s(\mathbb{R}^n)$.

%%%%%%%%%%%%%%%%%%%%%%%%%%%%%%%%%%%%%%%%%%%%%%%%%%%%%%%%%%%%%%%%%%%%%%%%%%%%%%%%%%%%%%
\section{The state equation}
\label{sub:stateequation}
%%%%%%%%%%%%%%%%%%%%%%%%%%%%%%%%%%%%%%%%%%%%%%%%%%%%%%%%%%%%%%%%%%%%%%%%%%%%%%%%%%%%%%
In this section, we derive the existence and uniqueness of a weak solution for the fractional heat equation \eqref{eq:fractional_heat}. In addition, we present an energy estimate and review regularity results.
%%%%%%%%%%%%%%%%%%%%%%%%%%%%%%%%%%%%%%%%%%%%%%%%%%%%%%%%%%%%%%%%%%%%%%%%%%%%%%%%%%%%%
\subsection{Eigenvalue problem}
\label{sec:eigenvalue_problem}
%%%%%%%%%%%%%%%%%%%%%%%%%%%%%%%%%%%%%%%%%%%%%%%%%%%%%%%%%%%%%%%%%%%%%%%%%%%%%%%%%%%%%%
Let us introduce the eigenvalue problem: Find 
%$(\lambda,\varphi) \in \mathbb{R} \times \tilde H^s(\Omega) \setminus \{ 0 \}$ such that
\begin{equation}
(\lambda,\varphi) \in \mathbb{R} \times \tilde H^s(\Omega) \setminus \{ 0 \}:
\qquad
\mathcal{A}(\varphi,v) = \lambda (\varphi,v)_{L^2(\Omega)} \quad \forall v \in \tilde H^s(\Omega).
\end{equation}
Spectral theory yields the existence of a countable collection of solutions $\{ \lambda_k,\varphi_k \} \subset \mathbb{R}^{+} \times \tilde H^s(\Omega)$ with the real eigenvalues enumerated in increasing order, counting multiplicities and such that $\{\varphi_k \}_{k \in \mathbb{N}}$ is an orthonormal basis of $L^2(\Omega)$ and an orthogonal basis of $\tilde H^s(\Omega)$.

%%%%%%%%%%%%%%%%%%%%%%%%%%%%%%%%%%%%%%%%%%%%%%%%%%%%%%%%%%%%%%%%%%%%%%%%%%%%%%%%%%%%%%
\subsection{Solution representation}
%%%%%%%%%%%%%%%%%%%%%%%%%%%%%%%%%%%%%%%%%%%%%%%%%%%%%%%%%%%%%%%%%%%%%%%%%%%%%%%%%%%%%%
We invoke the eigenparis $\{ \lambda_k,\varphi_k \}_{k \in \mathbb{N}}$, defined in \EO{the previous section}, and formally write the solution to problem \eqref{eq:fractional_heat} as
\begin{equation}
 \label{eq:solution_representation}
 \usf(x,t) = \sum_{k=1}^{\infty} \usf_k(t) \varphi_k(x).
\end{equation}
Since, at this formal stage, we have $\usf(x,0) = \usf_0(x)$,  this representation yields the following fractional initial value problem for $\usf_k$:
\begin{equation}
 \label{eq:u_k}
 \partial_t \usf_k(t) + \lambda_k \usf_k(t) = \fsf_k(t) + \zsf_k(t), \quad \usf_k(0) = \usf_{0,k},  \quad k \in \mathbb{N},
\end{equation}
where $\usf_{0,k} = ( \usf_0, \varphi_k)_{L^2(\Omega)}$, $\fsf_k(t) = (\fsf (\cdot,t),\varphi_k)_{L^2(\Omega)}$, and $\zsf_k(t) = (\zsf (\cdot,t),\varphi_k)_{L^2(\Omega)}$.
An explicit representation formula for the solution $\usf_k$ to problem \eqref{eq:u_k} holds:
\begin{equation}
\label{eq:u_k_solution}
\usf_k(t) = \usf_{k,0}e^{-\lambda_k t } + \int_{0}^{t} e^{-\lambda_k(t-r)} (\fsf_k(r) + \zsf_k(r)) \diff r.
\end{equation}

%%%%%%%%%%%%%%%%%%%%%%%%%%%%%%%%%%%%%%%%%%%%%%%%%%%%%%%%%%%%%%%%%%%%%%%%%%%%%%%%%%%%%%
\subsection{Well--posedness}
%%%%%%%%%%%%%%%%%%%%%%%%%%%%%%%%%%%%%%%%%%%%%%%%%%%%%%%%%%%%%%%%%%%%%%%%%%%%%%%%%%%%%%

A weak formulation for problem \eqref{eq:fractional_heat} reads as follows: Find $\usf \in \V$ such that $\usf(0) = \usf_0$ and, for a.e.~$t \in (0,T)$,
\begin{equation}
\label{eq:weak_formulation}
\langle \partial_t \usf, \phi \rangle + \mathcal{A}(\usf,\phi) 
  = \langle \fsf + \zsf,  \phi \rangle \qquad \forall \phi \in \tilde H^s(\Omega).
\end{equation}
The space $\mathbb{V}$ is defined as
\begin{equation}
 \label{eq:V}
 \mathbb{V} := \{ \vsf \in L^2(0,T;\tilde H^s(\Omega)) \cap L^{\infty}(0,T;L^2(\Omega)): \partial_t \vsf \in L^2(0,T;H^{-s}(\Omega)) \}.
\end{equation}

To simplify the exposition, we define
\begin{equation}
\label{eq:Sigma}
\Sigma^2 (\vsf,\gsf) := \| \vsf \|_{L^2(\Omega)}^2 +  \| \gsf \|^2_{L^2(0,T;H^{-s}(\Omega))}.
\end{equation}

\EO{We present the following existence and uniqueness result.}

\begin{theorem}[well--posedness of \eqref{eq:weak_formulation}]
\label{thm:exis_uniq}
Given $s \in (0,1)$, $\fsf \in L^2(0,T;H^{-s}(\Omega))$, $\zsf \in L^2(0,T;H^{-s}(\Omega))$, and $\usf_0 \in L^2(\Omega)$, problem \eqref{eq:weak_formulation} has a unique weak solution. In addition, we have the following energy estimate
\begin{equation}
\label{eq:energy_u}
\|\usf \|_{L^{^{\!\infty}\!}(0,T;L^2(\Omega))} + \|\usf \|_{L^2(0,T;H^s(\mathbb{R}^n))} \lesssim \Sigma (\usf_0,\fsf + \zsf).
\end{equation}
The hidden constant does not depend on $\usf$ nor the problem data.
\end{theorem}
\begin{proof}
Existence and uniqueness of a weak solution for problem \eqref{eq:fractional_heat} can be obtained in view of a standard spectral decomposition approach based on the solution representation \eqref{eq:solution_representation}. The energy estimate \eqref{eq:energy_u} also follows from such a spectral decomposition approach. 
% For brevity, we leave the details to the reader.
\end{proof}

Define
\begin{equation}
\label{eq:Upsilon}
\Upsilon^2 (\vsf,\gsf) := \| \vsf \|_{H^s(\mathbb{R}^n)}^2 +  \| \gsf \|^2_{L^2(0,T;L^2(\Omega))}.
\end{equation}

Let $\fsf \in L^2(0,T;L^2(\Omega))$, $\zsf \in L^2(0,T;L^2(\Omega))$, and $\usf_0 \in H^s(\mathbb{R}^n)$. Standard arguments, which heuristically entail multiplying the state equation \eqref{eq:fractional_heat} by the derivative of the solution $\usf$, yield the energy estimate
\begin{equation}
\label{eq:energy_u_2}
\| \partial_t \usf \|_{L^2(0,T;L^2(\Omega))} + \|\usf \|_{L^{\infty}(0,T;H^s(\mathbb{R}^n))} \lesssim \Upsilon (\usf_0,\fsf + \zsf),
\end{equation}
where the hidden constant does not depend on $\usf$ nor the problem data.
%%%%%%%%%%%%%%%%%%%%%%%%%%%%%%%%%%%%%%%%%%%%%%%%%%%%%%%%%%%%%%%%%%%%%%%%%%%%%%%%%%%%%%
\subsection{Regularity estimates}
%%%%%%%%%%%%%%%%%%%%%%%%%%%%%%%%%%%%%%%%%%%%%%%%%%%%%%%%%%%%%%%%%%%%%%%%%%%%%%%%%%%%%%

\EO{We present the following regularity result.

%\begin{theorem}[regularity estimate]
%Let $s \in (0,1)$ and $\Omega$ be a domain such that $\partial \Omega \in C^{\infty}$. If $\fsf + \zsf \in L^{\infty}(0,T;L^2(\Omega))$ and $\usf_0 \in \tilde H^s(\Omega)$, then
%\begin{equation}
% \label{eq:regularity_estimate}
% \| \usf \|_{L^2(0,T;H^{s+\gamma}(\Omega))} + \| \partial_t \usf \|_{L^2(Q)} \lesssim \| \usf_0\|_{H^s(\mathbb{R}^n)} + \| \fsf +\zsf\|_{L^{\infty}(0,T;L^2(\Omega))}, 
%\end{equation}
%where $\gamma = \min\{s, 1/2-\epsilon\}$. The hidden constant is independent of $\usf$ and the problem data.
%\label{th:regularity_estimate}
% \end{theorem}
%\begin{proof}
% See \cite[Theorems 3.1 and 3.2]{Francisco}.
%\end{proof}

\begin{theorem}[regularity estimate]
Let $s \in (0,1)$ and $\Omega$ be a domain such that $\partial \Omega \in C^{\infty}$. If $\fsf + \zsf \in L^{2}(0,T;H^r(\Omega))$ and $\usf_0 \in \tilde H^{s+r}(\Omega)$, for some $-s \leq  r < 1/2 - s$, then the solution $\usf$ of problem \eqref{eq:weak_formulation} belongs to $L^2(0,T;H^{s + \vartheta}(\Omega))$, where $\vartheta = \min \{ s+r, 1/2 - \epsilon\}$ and $\epsilon >0$ is arbitrarily small. In addition, we have 
\begin{equation}
 \label{eq:regularity_estimate}
 \| \usf \|_{L^2(0,T;H^{s+\vartheta}(\Omega))} \lesssim \| \usf_0\|_{H^{s+r}(\mathbb{R}^n)} + \| \fsf +\zsf\|_{L^{2}(0,T;H^r(\Omega))}.
\end{equation}
The hidden constant is independent of $\usf$ and the problem data.
\label{th:regularity_estimate}
 \end{theorem}
\begin{proof}
Assume that $\usf_0 \equiv 0$. We are thus in position to apply the results of \cite[Theorem 0.2]{MR3771838} and \cite[Theorem 5.8]{MR3880448} to conclude the desired regularity properties for $\usf$. Notice that we have also used the relations (3.20) and (5.19) in \cite{MR3880448} to relate the involved H\"{o}rmander spaces with Sobolev spaces. The case $\usf_0 \neq 0$ can be treated as in the proof of \cite[Theorem 3.2, Chapter 4]{MR0350178}. Indeed, since $\usf_0 \in \tilde H^{r+s}(\Omega)$ and  
\[\tilde H^{r+s}(\Omega)= [\tilde H^{r+2s}(\Omega),\tilde H^r(\Omega)]_{1/2}\] 
\cite[Theorem 11.6, Chapter 1]{Lions}, we invoke \cite[Theorem 3.2, Chapter 1]{Lions}, with $X =  \tilde H^{r+2s}(\Omega)$, $Y =  \tilde H^{r}(\Omega)$, $m=1$, and $j=0$, to conclude the existence of $\wsf$ such that $\wsf(0) = \usf_0$,
$
\wsf \in L^2(0;T; \tilde H^{r+2s}(\Omega)),
$
and
$
\partial_t \wsf \in L^2(0,T; \tilde H^{r}(\Omega)).
$
Define $\mathfrak{u}:= \usf - \wsf$ and observe that $\frak{u}$ satisfies
\[
\partial_t \mathfrak{u} + (-\Delta)^s \mathfrak{u} = \fsf + \zsf - \left( \partial_t \wsf + (-\Delta)^s \wsf \right)  \textrm{ in } \Omega \times (0,T),
\quad
\mathfrak{u}(0)= 0 \textrm{ in } \Omega,
\]
and $\mathfrak{u} = 0 \textrm{ in } \Omega^c \times (0,T)$. Since $\mathfrak{u}(0)= 0$ and $\partial_t \wsf + (-\Delta)^s \wsf \in L^2(0,T; H^{r}(\Omega))$ we can apply the results of \cite[Theorem 0.2]{MR3771838} and \cite[Theorem 5.8]{MR3880448} to conclude the desired regularity for $\mathfrak{u}$. Invoke, again, the properties that $\wsf$ satisfies to obtain \eqref{eq:regularity_estimate}. This concludes the proof.
%Indeed, since $\usf_0 \in \tilde H^{s+r}(\Omega)$, then there exists a $\wsf$ satisfying $\partial_t \wsf \in L^2((0,T);H^{r}(\Omega))$, $\wsf \in L^2((0,T);H^{s+ \vartheta}(\Omega))$, and $\wsf = \usf_0$ in $\Omega$. Now, notice that $\mathfrak{u}:= \usf - \wsf$ satisfies
%%Define $\wsf = \usf - \usf_0$. Notice that $\wsf$ solves
%\[
%\partial_t \mathfrak{u} + (-\Delta)^s \mathfrak{u} = \fsf + \zsf - \left(  (-\Delta)^s \wsf +\partial_t \wsf \right) \textrm{ in } \Omega \times (0,T),
%\quad
%\mathfrak{u}(0)= 0 \textrm{ in } \Omega.
%\]
%Upon extending $\usf_0$ appropriately, we have $\wsf = 0$ in $\Omega^c \times (0,T)$. 
%Since $\usf(0) = 0$.
\end{proof}
}

%%%%%%%%%%%%%%%%%%%%%%%%%%%%%%%%%%%%%%%%%%%%%%%%%%%%%%%%%%%%%%%%%%%%%%%%%%%%%%%%%%%%%%
\section{The fractional control problem}
\label{sec:control}
%%%%%%%%%%%%%%%%%%%%%%%%%%%%%%%%%%%%%%%%%%%%%%%%%%%%%%%%%%%%%%%%%%%%%%%%%%%%%%%%%%%%%%

In this section, we study the \emph{parabolic fractional optimal control problem}. We provide existence and uniqueness results together with first order necessary and sufficient optimality conditions.

The parabolic fractional optimal control problem reads: Find
$
 \text{min } J(\usf,\zsf) 
$
subject to the state equation \eqref{eq:fractional_heat} and the control constraints \eqref{eq:cc}. The set of \emph{admissible controls} is defined by
\begin{equation}
 \label{ac}
 \Zad:= \left\{ \wsf \in L^2(Q):\ \asf(x,t) \leq \wsf(x,t) \leq \bsf(x,t)~\textrm{a.e.}~(x,t) \in Q \right\}.
\end{equation}
Notice that $\Zad$ is a nonempty, bounded, closed, and convex subset of $L^2(Q)$. \EO{We assume that the desired state $\usf_d \in L^2(Q)$.}

To analyze \eqref{eq:min}--\eqref{eq:cc}, we introduce the so-called control to state operator.

\begin{definition}[control to state operator]
\label{def:fractional_operator}
The map $\mathbf{S}: L^2(0,T;H^{-s}(\Omega)) \ni \zsf \mapsto \usf(\zsf) \in \mathbb{V}$, where $\usf(\zsf)$ solves \eqref{eq:fractional_heat}, is called the fractional control to state operator.
\end{definition}

We immediately notice that the control to state operator $\mathbf{S}$ is affine. \EO{In fact, 
$
 \mathbf{S}(\zsf) = \mathbf{S}_0(\zsf) + \psi_0,
$
where} $\mathbf{S}_0(\zsf)$ denotes the solution to \eqref{eq:fractional_heat} with $\fsf=0$ and $\usf_0 = 0$, while $\psi_0$ solves \eqref{eq:fractional_heat} with $\zsf = 0$. Notice that $\mathbf{S}_0$ is linear and continuous. By the estimates of Theorem~\ref{thm:exis_uniq}, $\mathbf{S}$ is continuous as well. Since $\mathbb{V} \hookrightarrow L^2(Q) \hookrightarrow L^2(0,T;H^{-s}(\Omega))$, we may consider the operator $\mathbf{S}$ as acting from $L^2(Q)$ into itself. For simplicity, we keep the notation $\mathbf{S}$. 

We now define an optimal fractional state-control pair.
\begin{definition}[optimal fractional state-control pair]
A state-control pair $(\ousf(\ozsf),\ozsf) \in \mathbb{V} \times \Zad$ is called optimal for \eqref{eq:min}--\eqref{eq:cc} if $\ousf(\ozsf) = \mathbf{S} \ozsf$ and
\[
 J(\ousf(\ozsf),\ozsf) \leq J(\usf(\zsf),\zsf)
\]
for all $(\usf(\zsf),\zsf) \in \mathbb{V} \times \Zad$ such that $\usf(\zsf)= \mathbf{S} \zsf$.
\label{def:fractional_pair}
\end{definition}

The existence and uniqueness of an optimal state-control pair is as follows.

\begin{theorem}[existence and uniqueness]
\label{TH:optimal_control}
The optimal control problem \eqref{eq:min}--\eqref{eq:cc} has a unique solution $(\ousf(\ozsf), \ozsf) \in  \EO{\mathbb{V} \times \Zad}$.
\end{theorem}
\begin{proof}
Invoke $\mathbf{S}$ and reduce the optimal control problem \eqref{eq:min}--\eqref{eq:cc} to: Minimize
\begin{equation}
\label{f_opt}
\EO{j}(\zsf): = \frac12  \| \mathbf{S} \zsf - \usf_d \|^2_{L^2(Q)} + \frac\mu2 \|\zsf \|^2_{L^2(Q)}
\end{equation}
over $\Zad$. The strict convexity of \EO{$j$} is immediate ($\mu >0$). In addition, \EO{$j$} is weakly lower semicontinuous and $\Zad$ is  weakly sequentially compact. The direct method of the calculus of variations \cite[Theorem 5.51]{MR3026831} allows us to conclude.
\end{proof}

%%%%%%%%%%%%%%%%%%%%%%%%%%%%%%%%%%%%%%%%%%%%%%%%%%%%%%%%%%%%%%%%%%%%%%%%%%%%%%%%%%%%%%
\subsection{Optimality conditions}
\label{sub:optim}
%%%%%%%%%%%%%%%%%%%%%%%%%%%%%%%%%%%%%%%%%%%%%%%%%%%%%%%%%%%%%%%%%%%%%%%%%%%%%%%%%%%%%%
The following result is standard.

\begin{lemma}[variational inequality]
$\ozsf \in \Zad$ minimizes $f$ over $\Zad$ if and only if it solves the variational inequality
\begin{equation}
\label{vi}
  (\EO{j}'(\ozsf),\zsf - \ozsf )_{L^2(Q)} \geq 0
\end{equation}
for every $\zsf \in \Zad$.
\label{le:variational_inequality}
\end{lemma}
\begin{proof}
See \cite[Lemma 2.21]{Tbook}.
\end{proof}

To explore first order optimality conditions, we introduce the adjoint state.

\begin{definition}[fractional adjoint state]
\label{def:fractional_adjoint} 
The solution $\psf = \psf(\zsf) \in \mathbb{V}$ of
\begin{equation}
\label{fractional_adjoint}
  - \partial_{t} \psf + (-\Delta)^s \psf = \usf - \usf_d \ \mathrm{in} \ Q, \quad
  \psf = 0 \ \mathrm{in} \  \Omega^c \times (0,T), \quad
  \psf(T) = 0 \ \mathrm{in} \ \Omega,
\end{equation}
for $\zsf \in L^2(0,T;H^{-s}(\Omega))$, is called the fractional adjoint state associated to $\usf = \usf(\zsf)$.
\end{definition}

The following result is instrumental.

\begin{lemma}[auxiliary result]
\label{le:identity}
Let $\ozsf$ denote the optimal control for problem \eqref{eq:min}--\eqref{eq:cc} and $\ousf = \mathbf{S}\ozsf$. For every $\zsf \in \Zad$, we have
\begin{equation}
\label{identity}
  (\ousf - \usf_d,\usf - \ousf)_{ L^2(Q)} =  ( \opsf, \zsf - \ozsf)_{L^2(Q)},
\end{equation}
where $\usf = \mathbf{S}\zsf \in \mathbb{V}$ and $\psf = \psf(\zsf) \in \mathbb{V}$ solve problems \eqref{eq:weak_formulation} and \eqref{fractional_adjoint}, respectively.
\end{lemma}
\begin{proof}
Define $\chi:= \usf - \ousf \in \mathbb{V}$. Since $\usf$ solves \eqref{eq:weak_formulation} and $\ousf = \mathbf{S}\ozsf$, we obtain that $\chi(0) = 0$ in $\Omega$ and that, for a.e. $t \in (0,T)$,
\begin{equation}
\label{phi}
    \langle \partial_t \chi, \phi \rangle  + \mathcal{A}(\chi, \phi)  = (\zsf - \ozsf, \phi)_{L^2(\Omega)} 
    \quad \forall \phi \in \tilde H^s(\Omega).
\end{equation}
%Set $\phi = \opsf(t)$ in \eqref{phi} and integrate over time to arrive at the identity
%\[
% \int_{0}^T \left[ \langle \partial_t \chi,\opsf \rangle  + \mathcal{A}(\chi , \opsf ) \right] \diff t =  (\zsf - \ozsf, \opsf)_{L^2(Q)}.
%\]
%In view of the initial condition $\chi(0)=0$, the terminal condition $\opsf(T)=0$, and the symmetry of the bilinear form $\mathcal{A}$, an integration by parts formula yields
%\[
%\int_{0}^{T}\left[ - \langle \partial_{t} \opsf, \chi \rangle  + \mathcal{A}(\opsf,\chi)  \right] \diff t =  ( \opsf, \zsf - \ozsf)_{L^2(Q)}.
%\]
Set $\phi = \opsf(t)$ in \eqref{phi} and integrate over time.
%\[
% \int_{0}^T \left[ \langle \partial_t \chi,\opsf \rangle  + \mathcal{A}(\chi , \opsf ) \right] \diff t =  (\zsf - \ozsf, \opsf)_{L^2(Q)}.
%\]
In view of the initial condition $\chi(0)=0$, the terminal condition $\opsf(T)=0$, and the symmetry of the bilinear form $\mathcal{A}$, an integration by parts formula yields
\[
\int_{0}^{T}\left[ - \langle \partial_{t} \opsf, \chi \rangle  + \mathcal{A}(\opsf,\chi)  \right] \diff t =  ( \opsf, \zsf - \ozsf)_{L^2(Q)}.
\]

Now, set $\chi$ as a test function in the weak version of \eqref{fractional_adjoint} and integrate over time. These arguments allow us to arrive at
\[
 \int_{0}^T \left[- \langle \partial_{t} \opsf,\chi \rangle  + \mathcal{A}( \opsf , \chi) \right] \diff t = (\ousf-\usf_d,\usf - \ousf)_{ L^2(Q)}.
\]

The desired identity \eqref{identity} follows immediately from the derived expressions.
\end{proof}

We now prove necessary and sufficient optimality conditions for \eqref{eq:min}--\eqref{eq:cc}.

\begin{theorem}[first-order optimality conditions]
\label{TH:fractional_fooc}
$\ozsf \in \Zad$ is the optimal control of problem \eqref{eq:min}--\eqref{eq:cc} if and only if it solves the variational inequality
\begin{equation}
\label{viz}
\left( \mu \ozsf + \opsf, \zsf - \ozsf \right)_{L^2(Q)} \geq 0 \qquad \forall \zsf \in \Zad,
\end{equation}
where $\opsf = \opsf(\ozsf)$ solves \eqref{fractional_adjoint} with $\usf$ replaced by $\ousf$.
\end{theorem}
\begin{proof}
\EO{The inequality \ref{le:variational_inequality}  follows from combining the results of Lemmas \ref{le:variational_inequality} and \ref{le:identity}. We refer the reader to the proof of \cite[Theorem 3.19]{Tbook} for details.}
%We invoke the results of Lemma \ref{le:variational_inequality} to conclude that $\ozsf \in \Zad$ is optimal for problem \eqref{eq:min}--\eqref{eq:cc} if and only if 
%\[
%( \ousf - \usf_{d}, \mathbf{S}_0(\zsf - \ozsf) )_{L^2(Q)} + \mu ( \ozsf,  \zsf - \ozsf )_{L^2(Q)} \geq 0.
%\]
%We recall that the control to state map $\mathbf{S}$ is affine \eqref{eq:S_affine}. Notice that $\mathbf{S}_0(\zsf - \ozsf) = \mathbf{S}_0 \zsf + \psi_0 - (\psi_0 + \mathbf{S}_0 \ozsf) = \usf(\zsf) - \ousf$. Consequently, 
%\[
%( \ousf - \usf_{d}, \usf(\zsf) - \ousf )_{L^2(Q)} + \mu ( \ozsf,  \zsf - \ozsf )_{L^2(Q)} \geq 0.
%\]
%Finally, we invoke \eqref{identity} to arrive at the desired variational inequality:
%\begin{align*}
%( \opsf,  \zsf - \ozsf )_{L^2(Q)} + \mu ( \ozsf,  \zsf - \ozsf )_{L^2(Q)} \geq 0.
%\end{align*}
%This concludes the proof.
\end{proof}

\subsection{Regularity \EO{estimates}}
\label{sub:regularity}
 %%%%%%%%%%%%%%%%%%%%%%%%%%%%%%%%%%%%%%%%%%%%%%%%%%%%%%%%%%%%%%%%%%%%%%%%%%%%%%%%%%%%%%%%%%%%%%%%%%%%%%%%%%%%%%%%%%%%%%%%%%%%%%%%%%%%%%%%%%%%%%%%%
In this section we derive regularity estimates \EO{for the optimal variables.} To accomplish this task, we recall the projection formula
\begin{equation}
\label{eq:projection_formula}
\bar \zsf  = \proj_{[\asf,\bsf]}\left(-\frac{1}{\mu}\bar \psf\right),
\quad
\proj_{[\asf,\bsf]}\left( \vsf \right):=  \min \left \{ \bsf, \max \left \{ \asf, \vsf \right \}  \right \}, 
\end{equation}
and refer the reader to \cite[section 3.6.3]{Tbook} for a proof of this result.

%To simply the exposition, we define
%\begin{equation}
%\label{eq:A}
%\mathfrak{A} := 
%%\| \usf_0 \|_{L^2(\Omega)} + \| \fsf \|_{L^2(0,T; H^{-s}(\Omega))} 
%\Sigma(\usf_0,\fsf)
%+ \| \usf_d \|_{L^2(Q)} 
%+  \| \asf \|_{H^1(0,T;L^2(\Omega))} + \| \bsf \|_{H^1(0,T;L^2(\Omega))},
%\end{equation}
%where $\Sigma$ is defined as in \eqref{eq:Sigma}.

We begin by deriving regularity estimates in time. 

\begin{theorem}[time regularity estimates]
Let $s \in (0,1)$, $\fsf \in L^2(0,T;H^{-s}(\Omega))$, and $\usf_0 \in L^2(\Omega)$. If $\asf, \bsf \in H^1(0,T;L^2(\Omega))$, \EO{then
\begin{multline}
\| \partial_t \bar \zsf \|_{L^2(Q)} 
+
\| \partial_t \bar \psf \|_{L^2(Q)} 
 \lesssim 
 \Sigma(\usf_0,\fsf)
+ \| \usf_d \|_{L^2(Q)} 
\\
+  \| \asf \|_{H^1(0,T;L^2(\Omega))} + \| \bsf \|_{H^1(0,T;L^2(\Omega))}.
 \label{eq:regularity_time}
\end{multline}
The} hidden constant is independent of the problem data and the optimal variables.
\label{th:regularity_time}
\end{theorem}
\begin{proof}
Since $\fsf + \bar \zsf \in L^2(0,T; H^{-s}(\Omega))$ and $\usf_0 \in L^2(\Omega)$, an application of the energy estimate \eqref{eq:energy_u} yields
\begin{equation}
\| \bar \usf \|_{L^{\infty}(0,T;L^2(\Omega))} + \|\bar \usf \|_{L^{2}(0,T;H^s(\mathbb{R}^n))} \lesssim \Sigma (\usf_0,\fsf + \bar \zsf).
\label{eq:energy_aux}
\end{equation}
On the other hand, since $\bar \usf - \usf_d \in L^2(0,T;L^2(\Omega))$, we apply the energy estimate \eqref{eq:energy_u_2} for the problem that $\bar \psf$ solves, i.e., problem \eqref{fractional_adjoint} with $\usf$ replaced by $\bar \usf$, to arrive at
\begin{equation}
\| \partial_t \bar \psf \|_{L^2(Q)} + \|\bar \psf \|_{L^{\infty}(0,T;H^s(\mathbb{R}^n))} \lesssim \| \bar \usf - \usf_d \|_{L^2(Q)},
\end{equation}
which, in view of \eqref{eq:energy_aux}, yields
\[
\| \partial_t \bar \psf \|_{L^2(Q)} + \|\bar \psf \|_{L^{\infty}(0,T;H^s(\mathbb{R}^n))} \lesssim
\| \usf_0 \|_{L^2(\Omega)} 
+ \| \fsf + \bar \zsf \|_{L^2(0,T;H^{-s}(\Omega))} + \| \usf_d \|_{L^2(Q)}.
\]
On the basis of this bound, we invoke the projection formula \eqref{eq:projection_formula} and \cite[Theorem A.1]{KSbook} to obtain $\bar \zsf \in H^1(0,T;L^2(\Omega))$ and \eqref{eq:regularity_time}. This concludes the proof.
\end{proof}

%Notice that the arguments elaborated in the proof of Theorem \ref{th:regularity_time} allow us to conclude that $\opsf \in H ^1(0,T;L^2(\Omega))$ with the estimate 

\EO{Before proceeding with the study of regularity estimate in space, we present the following instrumental result.

\begin{lemma}[nonlinear interpolation]
\label{le:nonlinear_interpolation}
Let $\mathcal{G}: L^2(\Omega) \rightarrow L^2(\Omega)$ be the nonlinear map defined as $\mathcal{G} \wsf = \max \{ \wsf,0\}$. If $s \in (0,1)$, then $\mathcal{G}$ maps $H^s(\Omega)$ into $H^s(\Omega)$ and
\[
\| \mathcal{G} \wsf \|_{H^s(\Omega)} \lesssim \| \wsf \|_{H^s(\Omega)} 
\]
for every $\wsf \in H^s(\Omega)$.
\end{lemma}
\begin{proof}
Notice that \cite[Theorem A.1]{KSbook} immediately yields the boundedness of $\mathcal{G}$ in $H^1(\Omega)$. On the other hand, we have the following Lipschitz property of $\mathcal{G}$ in $L^2(\Omega)$:
\[
\| \mathcal{G} \wsf_1 - \mathcal{G} \wsf_2 \|_{L^2(\Omega)} \leq \| \wsf_1 - \wsf_2 \|_{L^2(\Omega)} \quad \forall \wsf_1, \wsf_2 \in L^2(\Omega).
\]
Since $H^s(\Omega) = [L^2(\Omega),H^1(\Omega)]_s$, we apply \cite[Lemma 28.1]{Tartar} to conclude.
\end{proof}

To present regularity estimates in space, we define
%%\begin{equation}
%%\label{eq:B}
%%\mathfrak{B} := 
%%%\|\usf_0\|_{L^2(\Omega)} + \| \fsf \|_{L^2(0,T;H^{-s}(\Omega))} 
%%\Sigma(\usf_0,\fsf)
%%+  \EO{\| \usf_d \|_{L^{2}(0,T;L^2(\Omega))}}
%%+ \| \asf \|_{L^2(0,T;H^1(\Omega))} + \| \bsf \|_{L^2(0,T;H^1(\Omega))}
%%\end{equation}
%and
\begin{multline}
\label{eq:C}
\mathfrak{C} := \|\usf_0\|_{\tilde H^{\beta}(\mathbb{R}^n)} + \| \fsf \|_{L^{2}(0,T;H^{\beta}(\Omega))} + \| \usf_d \|_{L^{2}(0,T;H^{\beta}(\Omega))} 
\\
+ \| \asf \|_{L^{2}(0,T;H^1(\Omega))} + \| \bsf \|_{L^{2}(0,T;H^1(\Omega))},
\end{multline}
where $\beta = 1 - \epsilon$ and $\epsilon > 0$ is arbitrarily small.}

\begin{theorem}[space regularity estimates: $s \in (0,1)$]
Let $s \in (0,1)$, $\Omega$ be a domain such that $\partial \Omega \in C^{\infty}$, and $\asf, \bsf \in L^2(0,T;H^1(\Omega))$. \EO{If $\usf_0 \in \tilde H^{\beta}(\Omega)$, $\fsf \in L^{2}(0,T;H^{\beta}(\Omega))$, and  $\usf_d \in L^{2}(0,T;H^{\beta}(\Omega))$, for every $\beta < 1$, then
\begin{align}
\label{eq:regularity_space_control_bigger_05}
 \| \bar \zsf \|_{L^2(0,T;H^1(\Omega))} +  \| \bar \psf \|_{L^2(0,T;H^{s+\frac12 - \epsilon}(\Omega))} 
+
\| \bar \usf \|_{L^2(0,T;H^{s+\frac12 - \epsilon}(\Omega))} 
 & \lesssim \mathfrak{C},
 \quad s > \frac{1}{2},
  \\
 \label{eq:regularity_space_control_05}
 \| \bar \zsf \|_{L^2(0,T;H^{1-\epsilon}(\Omega))} +  \| \bar \psf \|_{L^2(0,T;H^{1-\epsilon}(\Omega))} + \| \bar \usf \|_{L^2(0,T;H^{1-\epsilon}(\Omega))} & \lesssim \mathfrak{C}, \quad s = \frac{1}{2},
 \end{align}
 and
 \begin{multline}
 \| \bar \zsf \|_{L^2(0,T;H^{s+\frac12-\epsilon}(\Omega))} +  \| \bar \psf \|_{L^2(0,T;H^{s+\frac12-\epsilon}(\Omega))} 
 \\
 \label{eq:regularity_space_control_smaller_05}
 + \| \bar \usf \|_{L^2(0,T;H^{s+\frac12-\epsilon}(\Omega))} 
 \lesssim \mathfrak{C}, \quad s < \frac{1}{2},
\end{multline}
where} $\epsilon>0$ is arbitrarily small. In all the estimates the hidden constant is independent of the optimal variables and the problem data.
\label{thm:regularity_space}
\end{theorem}
\begin{proof}
\EO{Let $s \in (0,1)$. Since the right-hand side of the state equation \eqref{eq:weak_formulation} is such that $\fsf + \bar \zsf \in L^2(Q)$, we can apply the energy estimate \eqref{eq:energy_u} to obtain that $\ousf \in L^{2}(0,T;H^s(\mathbb{R}^n))$. By assumption $\usf_d \in L^2(0,T;H^{\beta}(\Omega))$ for every $\beta < 1$. Consequently, $\ousf - \usf_d \in L^{2}(0,T;H^{s}(\Omega))$. Estimate \eqref{eq:regularity_estimate}, with $r=s$, thus yields
\begin{equation}
\label{eq:regularity_space_adjoint_bigger_05}
 \| \opsf \|_{L^2(0,T;H^{s+\iota}(\Omega))}  \lesssim \| \ousf - \usf_d \|_{L^{2}(0,T;H^s(\Omega))},
\end{equation}
where $\iota = \min \{2s,1/2-\epsilon\}$ and $\epsilon>0$ is arbitrarily small. 

Similarly, since $s<1$, we have that $\usf_0 \in \tilde H^s(\Omega)$. By assumption, we also have that $\fsf \in L^{2}(Q)$ and $\asf, \bsf \in L^{2}(Q)$. We are thus in position to apply Theorem \ref{th:regularity_estimate}, with $r=0$, to obtain that $\bar \usf \in L^2(0,T;H^{s+\nu}(\Omega))$, where $\nu = \min \{ s,1/2 - \epsilon\}$, and 
\begin{equation}
\label{eq:regularity_space_state_bigger_05_aux}
 \| \ousf \|_{L^2(0,T;H^{s+\nu}(\Omega))}  \lesssim \| \usf_0 \|_{H^s(\mathbb{R}^n)}+ \| \fsf + \ozsf \|_{L^{2}(0,T;L^2(\Omega))}.
\end{equation} 

We now consider four cases.

\noindent 
\boxed{1} $s \in (\tfrac{1}{2},1)$: Since $s > 1/2$, \eqref{eq:regularity_space_adjoint_bigger_05} immediately yields $\opsf \in L^2(0,T;H^{s+1/2-\epsilon}(\Omega))$ for $\epsilon >0$ arbitrarily small. Thus, in view of the projection formula \eqref{eq:projection_formula}, we obtain
\begin{multline*}
  \| \ozsf \|_{L^2(0,T;H^{1}(\Omega))} \lesssim \| \ousf - \usf_d \|_{L^{2}(0,T;H^{s}(\Omega))} + \| \asf \|_{L^2(0,T;H^1(\Omega))} + \| \bsf \|_{L^2(0,T;H^1(\Omega))}
  \\
  \lesssim \Sigma(\usf_0, \fsf + \ozsf) + \| \usf_d \|_{L^{2}(0,T;H^s(\Omega))} + \| \asf \|_{L^2(0,T;H^1(\Omega))} + \| \bsf \|_{L^2(0,T;H^1(\Omega))} \lesssim \mathfrak{C}.
\end{multline*}
Notice that we have also used \eqref{eq:energy_u}. On the other hand, \eqref{eq:regularity_space_state_bigger_05_aux} immediately implies that $\bar \usf \in L^2(0,T; H^{s+1/2-\epsilon}(\Omega))$, for $\epsilon >0$ arbitrarily small, together with the bound
\begin{multline*}
  \| \ousf \|_{L^2(0,T;H^{s+1/2-\epsilon}(\Omega))} \lesssim \| \usf_0 \|_{H^s(\mathbb{R}^n)}+ \| \fsf \|_{L^{2}(0,T;L^2(\Omega))}
  \\
  + \| \asf \|_{L^2(0,T;L^2(\Omega))} + \| \bsf \|_{L^2(0,T;L^2(\Omega))} \lesssim \mathfrak{C}.
\end{multline*}
A collection of the derived estimates yields \eqref{eq:regularity_space_control_bigger_05}.

\noindent 
\boxed{2} $s= \tfrac{1}{2}$:  The proof of the estimate \eqref{eq:regularity_space_control_05} follows similar arguments. For brevity, we skip the details.

\noindent 
\linebreak
\boxed{3} $s \in [\tfrac{1}{4},\tfrac{1}{2})$: Estimate \eqref{eq:regularity_space_adjoint_bigger_05} yields $\opsf \in L^2(0,T;H^{s + 1/2 - \epsilon}(\Omega))$, for $\epsilon >0$ arbitrarily small, because $\iota = \min \{ 2s, 1/2 - \epsilon \} = 1/2 - \epsilon$. We can thus apply the nonlinear argument of Lemma \ref{le:nonlinear_interpolation} to arrive at the bound
\begin{multline}
  \| \ozsf \|_{L^2(0,T;H^{s + 1/2 - \epsilon}(\Omega))} \lesssim 
    \|\opsf\|_{L^2(0,T;H^{s + 1/2 - \epsilon}(\Omega))}  + \| \asf \|_{L^2(0,T;H^1(\Omega))} + \| \bsf \|_{L^2(0,T;H^1(\Omega))}
  \\
    \lesssim \| \ousf - \usf_d \|_{L^2(0,T;H^{s}(\Omega))} 
  + \| \asf \|_{L^2(0,T;H^1(\Omega))} + \| \bsf \|_{L^2(0,T;H^1(\Omega))} \lesssim \mathfrak{C}.
\end{multline}
Since $s < 1/2$, by assumption we have that $\usf_0 \in \tilde H^{2s}(\Omega)$. Notice that we also have that $\fsf + \ozsf \in L^{2}(0,T;H^{s}(\Omega))$ because $s + 1/2 - \epsilon > s$ for $\epsilon>0$ arbitrarily small. Invoke thus the regularity estimate \eqref{eq:regularity_estimate}, with $r=s$, to arrive at
\begin{equation}
  \| \ousf \|_{L^2(0,T;H^{s + \vartheta}(\Omega))} \lesssim  \|\usf_0 \|_{H^{2s}(\mathbb{R}^n)} + \| \fsf +\ozsf\|_{L^{2}(0,T;H^{s}(\Omega))} \lesssim \mathfrak{C},
  \label{eq:bound_u_14_12}
\end{equation}
where $\vartheta = \min \{ 2s, 1/2 - \epsilon \}$ and $\epsilon >0$. Since $s \geq 1/4$, we have that $\vartheta = 1/2 - \epsilon$. Consequently, $\bar \usf \in L^2(0,T; H^{s+1/2 - \epsilon}(\Omega))$, for every $\epsilon>0$, with the estimate \eqref{eq:bound_u_14_12}.

%%Now, observe that $\bar \usf - \usf_d \in L^2((0,T); H^{s+1/2-\epsilon}(\Omega))$, for every $\epsilon >0$. We can thus invoke Theorem \ref{th:regularity_estimate}, with $r = s+1/2-\epsilon$, to derive the bound
%%\begin{equation}
%%\label{eq:regularity_space_adjoint_14_12}
%% \| \opsf \|_{L^2(0,T;H^{s+\vartheta}(\Omega))} \lesssim \| \ousf - \usf_d \|_{L^{2}(0,T;H^{s+1/2 - \epsilon}(\Omega))},
%%\end{equation}
%%where $\vartheta = \min \{ 2s + 1/2 - \epsilon, 1/2 - \epsilon \} = 1/2 - \epsilon$. This immediately yields that $\opsf \in L^{2}(0,T;H^{s+1/2 - \epsilon}(\Omega))$ with the bound \eqref{eq:regularity_space_adjoint_14_12}. To obtain regularity properties for the optimal control variable $\bar \zsf$ we invoke Lemma \ref{le:nonlinear_interpolation}. In fact, we have
%%\begin{multline}
%%  \| \ozsf \|_{L^2(0,T;H^{s+1/2 - \epsilon}(\Omega))} \lesssim \|\usf_0 \|_{H^{2s}(\mathbb{R}^n)} + \| \fsf +\ozsf\|_{L^{2}(0,T;H^{s}(\Omega))} 
%%  \\
%%  + \| \usf_d \|_{L^{2}(0,T;H^{s+1/2 - \epsilon}(\Omega))}
%%  + \| \asf \|_{L^2(0,T;H^1(\Omega))} + \| \bsf \|_{L^2(0,T;H^1(\Omega))} \lesssim \mathfrak{C}.
%%\end{multline}

\noindent 
\linebreak
\boxed{4} $s\in (0,\tfrac{1}{4})$: We proceed on the basis of a bootstrap argument. Since $s<1/4$, we have that $\iota = 2s$ and, as a consequence, $\opsf \in L^2(0,T;H^{3s}(\Omega))$. We thus invoke the nonlinear argument of Lemma \ref{le:nonlinear_interpolation} to conclude that $\ozsf \in L^2(0,T;H^{3s}(\Omega))$ with
\begin{multline}
  \| \ozsf \|_{L^2(0,T;H^{3s}(\Omega))} \lesssim \| \ousf \|_{L^{2}(0,T;H^{s}(\Omega))} + \|  \usf_d \|_{L^{2}(0,T;H^{s}(\Omega))}
  \\
+
  \| \asf \|_{L^2(0,T;H^1(\Omega))} + \| \bsf \|_{L^2(0,T;H^1(\Omega))} \lesssim \mathfrak{C}.
\end{multline}
Now, since $\fsf + \ozsf \in L^2(0,T;H^{3s}(\Omega))$ and $\usf_0 \in \tilde H^{4s}(\Omega)$, we apply Theorem \ref{th:regularity_estimate}, with $r =3s$, to conclude that
\begin{equation}
  \| \ousf \|_{L^2(0,T;H^{s + \upsilon}(\Omega))} \lesssim \|\usf_0 \|_{H^{4s}(\mathbb{R}^n)} + \| \fsf +\ozsf\|_{L^{2}(0,T;H^{3s}(\Omega))},
\end{equation}
where $\upsilon = \min \{ 4s, 1/2 - \epsilon \}$ and $\epsilon > 0$ is arbitrarily small.

\noindent 
\linebreak
\boxed{4.1} $s\in [\tfrac{1}{8},\tfrac{1}{4})$: We immediately conclude that $\ousf \in L^2(0,T;H^{s+1/2-\epsilon}(\Omega))$, for $\epsilon >0$ arbitrarily small.  Invoke Theorem \ref{th:regularity_estimate}, with $r = s+1/2-\epsilon$, to conclude that
\[
 \| \opsf \|_{L^2(0,T;H^{s+\theta}(\Omega))}  \lesssim \| \ousf - \usf_d \|_{L^{2}(0,T;H^{s+1/2-\epsilon}(\Omega))} \lesssim \mathfrak{C},
\]
where $\vartheta = \min \{2s+1/2 - \epsilon, 1/2 - \epsilon \} = 1/2 - \epsilon$. This, in view of Lemma \ref{le:nonlinear_interpolation}, yields $\ozsf \in L^2(0,T;H^{s+1/2-\epsilon}(\Omega))$, for $\epsilon > 0$ arbitrarily small, with a similar estimate.

\noindent 
\linebreak
\boxed{4.2} $s\in (0,\tfrac{1}{8})$: Here, $\upsilon = 4s$. Thus $\ousf \in L^2(0,T; H^{5s}(\Omega))$. Invoke Theorem \ref{th:regularity_estimate}, with $r = 5s$, to conclude that
\[
 \| \opsf \|_{L^2(0,T;H^{s+\sigma}(\Omega))}  \lesssim \| \ousf - \usf_d \|_{L^{2}(0,T;H^{5s}(\Omega))} \lesssim \mathfrak{C},
\]
where $\sigma= \min \{ 6s, 1/2 - \epsilon \}$, for $\epsilon > 0$ arbitrarily small.

\noindent 
\linebreak
\boxed{4.2.1} $s\in [\tfrac{1}{12},\tfrac{1}{8})$: We have that $\opsf, \ozsf  \in L^2(0,T; H^{s+1/2 - \epsilon}(\Omega))$, for $\epsilon>0$ arbitrarily small. Invoke Theorem \ref{th:regularity_estimate}, with $r = s + 1/2 - \epsilon$, to arrive at
\[
 \| \ousf \|_{L^2(0,T;H^{s+\theta}(\Omega))}  \lesssim \| \fsf + \ozsf \|_{L^{2}(0,T;H^{s+1/2 -\epsilon}(\Omega))} \lesssim \mathfrak{C},
\]
where $\vartheta = \min \{ 2s + 1/2 - \epsilon, 1/2 - \epsilon \} = 1/2 - \epsilon$, i.e., $\ousf \in L^2(0,T;H^{s-1/2 - \epsilon}(\Omega))$, for $\epsilon > 0$ arbitrarily small.

\noindent 
\linebreak
\boxed{4.2.2} $s\in (0,\tfrac{1}{12})$: Here, $\sigma = 6s$ and thus $\opsf \in L^2(0,T;H^{7s}(\Omega))$. Argue as before to conclude that $\ozsf \in L^2(0,T;H^{7s}(\Omega))$ and $\ousf \in L^2(0,T;H^{s + \vartheta}(\Omega))$, where $\vartheta = \min\{ 8s, 1/2 - \epsilon\}$, for $\epsilon >0$ arbitrarily small.

From this procedure we note that, at every step, there is a regularity gain. Consequently, after a finite number of steps, which is proportional to $s^{-1}$, we can conclude that estimate \eqref{eq:regularity_space_control_smaller_05} holds. This concludes the proof.}
\end{proof}

\section{Approximation of the state equation}
 \label{sec:a_priori_state}
 %%%%%%%%%%%%%%%%%%%%%%%%%%%%%%%%%%%%%%%%%%%%%%%%%%%%%%%%%%%%%%%%%%%%%%%%%%%%%%%%%%%%%%%%%%%%%%%%%%%%%%%%%%%%%%%%%%%%%%%%%%%%%%%%%%%%%%%%%%%%%%%%%
 
Let us now propose and analyze a fully discrete scheme to solve the state equation \eqref{eq:weak_formulation}. The space discretization hinges on a standard finite element space of continuous and piecewise linear functions. The discretization in time uses the backward Euler scheme. We derive stability estimates and an a priori error estimate in $L^2(Q)$ 
 
 %%%%%%%%%%%%%%%%%%%%%%%%%%%%%%%%%%%%%%%%%%%%%%%%%%%%%%%%%%%%%%%%%%%%%%%%%%%%%%%%%%%%%%
 \subsection{Time discretization}
 \label{sec:time_discretization}
 %%%%%%%%%%%%%%%%%%%%%%%%%%%%%%%%%%%%%%%%%%%%%%%%%%%%%%%%%%%%%%%%%%%%%%%%%%%%%%%%%%%%%%
 Let $\K \in \mathbb{N}$ be the number of time steps. Define the uniform time step $\tau = T/\K > 0$ and set $t_k = k \tau$ for $k=0,\dots,\K$. We denote the time partition by $\mathcal{T}:= \{ t_k \}_{k=0}^{\K}$. Given a function $\phi \in C ( [0,T], {\Xcal} )$, we denote $\phi^k = \phi(t_k) \in \Xcal$ and $\phi^{\tau}= \{ \phi^k\}_{k=0}^{\K} \subset \Xcal$. For any sequence $\phi^{\tau} \subset \Xcal$, we define the piecewise linear interpolant ${\hat \phi}^{\tau} \in C([0,T];\Xcal)$ by
 \begin{equation}
 \label{eq:p_l_interpolant}
  {\hat \phi}^{\tau}(t):= \frac{t-t_k}{\tau}\phi^{k+1} +  \frac{t_{k+1}-t}{\tau}\phi^{k}, \quad t \in [t_k,t_{k+1}], \quad k= 0, \dots, \mathcal{K} - 1.
 \end{equation}
We also define, for any sequence $\phi^{\tau} \subset \Xcal$, the first order differences operators
 \begin{equation}
 \label{1_discrete}
   \mathfrak{d} \phi^{k+1} = \tau^{-1} (\phi^{k+1} - \phi^{k}), \quad k = 0,\ldots, \K -1,
 \end{equation}
 and
\begin{equation}
\label{1_discrete_forward}
\bar{\mathfrak{d}} \phi^k = -\tau^{-1} \left( \phi^{k+1} - \phi^k \right), \quad k = \K -1, \ldots, 0,
\end{equation}
and the norms $\| \phi^{\tau} \|_{\ell^{\infty}(\Xcal)} = \max \{ \| \phi^k\|_{\Xcal} :  k = 0, \dots, \K \}$ and  
 \[\| \phi^{\tau} \|_{\ell^p(\Xcal)} = \left( \sum_{k=1}^\K \tau \| \phi^k\|_{\Xcal}^p \right)^{\frac{1}{p}}, \quad p \in [1,\infty),
 \]
\EO{Finally, we define $| \phi^{\tau} |_{\ell^{\infty}(\Xcal)} = \max \{ | \phi^k|_{\Xcal} :  k = 0, \dots, \K \}$;  $| \phi^{\tau} |_{\ell^p(\Xcal)}$ is defined accordingly.}
% \[ | \phi^{\tau} |_{\ell^p(\Xcal)} = \left( \sum_{k=1}^\K \tau | \phi^k|_{\Xcal}^p \right)^{\frac{1}{p}}, \quad p \in [1,\infty).\]
% }

\begin{remark}[identification with a piecewise constant function]\rm
\label{rk:indentification}
We note that any sequence $\phi^{\tau} \subset \Xcal$ can be equivalently understood as a piecewise
constant, in time, function $\phi \in L^{\infty}(0,T;\Xcal)$. \EO{Indeed, let us consider}
\[
\phi(t) = \phi^k \quad \forall t \in (t_{k-1},t_k], \quad \ k = 1,\dots, \K.
\]
In what follows we use this identification repeatedly and without explicit mention.
\end{remark}
 
\subsection{Space discretization}
\label{sec:space_discretization}

Let $\T = \{ K \}$ be a conforming partition of $\overline \Omega$ into simplices $K$ with size $h_K = \diam(K)$. Set $h_{\T} = \max_{K \in \T} h_K$. We denote by $\mathbb{T}$ the collection of conforming and shape regular meshes that are refinements of an initial mesh $\T_0$. By shape regular we mean that there exists a constant $\sigma > 1$ such that $\max \{ \sigma_K: K \in \T \} \leq \sigma$ for all $\T \in \mathbb{T}$. Here, $\sigma_K = h_K / \rho_K$ is the shape coefficient of $K$; $\rho_K$ denotes the diameter of the largest ball that can be inscribed in $K$ \cite{MR1930132}.

Given a mesh $\T \in \mathbb{T}$, we define the finite element space of continuous piecewise polynomials of degree one as
\begin{equation}
\V(\T) = \left\{ V  \in C^0( \overline\Omega): \EO{V|_{K}} \in \mathbb{P}_1(K) \ \forall K \in \T, \ V = 0 \textrm{ on } \partial \Omega \right\}.
\label{eq:defFESpace}
\end{equation}
Note that discrete functions are trivially extended by zero to $\Omega^c$ and that we enforce a classical homogeneous Dirichlet boundary condition at the degrees of freedom that are located at the boundary of \(\Omega\).

\subsection{Elliptic projection}
\label{sec:elliptic_projection}
In this section, we define an elliptic projector that will be of fundamental importance to derive error estimates. This projector $G_{\T}: \tilde H^s(\Omega) \rightarrow \V(\T)$ is such that, for $w \in \tilde H^s(\Omega)$, it is given by 
\begin{equation}
\label{eq:elliptic_projection}
G_{\T}w \in \V(\T): \quad \mathcal{A}(G_{\T}w,W) = \mathcal{A}(w,W) \quad \forall W \in \V(\T).
\end{equation}

The operator $G_{\T}$ satisfies the following stability and approximation properties.

\begin{proposition}[elliptic projector]
Let $s \in (0,1)$. The elliptic projector $G_{\T}$ is stable in $\tilde H^s(\Omega)$, i.e.,
\begin{equation}
\| G_{\T} w \|_s \lesssim \| w \|_s \quad \forall w \in \tilde H^s(\Omega).
\end{equation}
If, in addition, 
% $\Omega$ is a domain such that $\partial \Omega \in C^{\infty}$ and 
$w \in H^{\kappa}(\Omega)$, for $\kappa \geq s$, then $G_{\T}$ has the approximation property
\begin{equation}
\label{eq:error_estimate_for_G}
 \| w -  G_{\T} w \|_{s} \lesssim h_{\T}^{\kappa-s} | w |_{H^{\kappa}(\Omega)}.
\end{equation}
In both estimates the hidden constants are independent of $w$ and $h_{\T}$.
\end{proposition}
\begin{proof}
\EO{The proof follows standard arguments. For brevity, we skip the details.}
%To show stability set $W = G_{\T} w$ in \eqref{eq:elliptic_projection}, invoke the definition of the norm $\| \cdot \|_s$ given by \eqref{eq:norm_s}, and utilize the continuity of $\mathcal{A}$.
%
%Obtaining the estimate \eqref{eq:error_estimate_for_G} hinges on Galerkin orthogonality: If $\Pi_{\T}$ denotes the Scott--Zhang quasi-interpolation operator, then
%\[
%  \| w -  G_{\T} w \|^2_{s} = \mathcal{A}(w -  G_{\T} w,w -  \Pi_{\T} w) \lesssim \| w -  G_{\T} w \|_s \| w -  \Pi_{\T} w \|_s.
%\]
%The assertion thus follows from an interpolation error estimate for $\Pi_{\T}$:
%\[
% \| w -  \Pi_{\T} w \|_s \lesssim h_{\T}^{\kappa - s} |w|_{H^{\kappa}(\Omega)};
%\]
%see \cite[Section 4.2]{MR3620141} for details. This concludes the proof.
\end{proof}

\begin{proposition}[$L^2(\Omega)$-error estimate: elliptic projector]
Let $s \in (0,1)$ and $\Omega$ be a domain such that $\partial \Omega \in C^{\infty}$. If $w \in H^{\kappa}(\Omega)$, for $\kappa \geq s$, then we have
\begin{equation}
\label{eq:L2_error_estimate_for_G}
 \| w -  G_{\T} w \|_{L^2(\Omega)} \lesssim h_{\T}^{\kappa + \vartheta - s} |w|_{H^{\kappa}(\Omega)},
\end{equation}
where \EO{$\vartheta = \min \{ s , 1/2 - \epsilon\}$}. The hidden constant is independent of $w$ and $h_{\T}$.
\label{pro:L2_elliptic}
\end{proposition}
\begin{proof}
 To obtain \eqref{eq:L2_error_estimate_for_G} we argue by duality. Let $z \in \tilde H^s(\Omega)$ be the solution to
\[
 \mathcal{A}(\phi,z) = \langle w -  G_{\T} w, \phi \rangle \quad \forall \phi \in \tilde H^s(\Omega).
\]
Set $\phi  = w - G_{\T} w$ and utilize that $ \mathcal{A}(w -  G_{\T} w, \Pi_{\T} z) = 0$, where $\Pi_{\T}$ denotes the Scott--Zhang quasi-interpolation operator, to obtain 
\[
 \| w -  G_{\T} w \|^2_{L^2(\Omega)} =  \mathcal{A}(w -  G_{\T} w,z) \leq \| w -  G_{\T} w \|_s \|z - \Pi_{\T} z \|_s .
\]
\EO{Apply an interpolation estimate for $\Pi_{\T}$ and Proposition \ref{pro:state_regularity_smooth}, with $r = 0$, to obtain}
%Invoke an interpolation error estimate for $\Pi_{\T}$ combined with the regularity results of Proposition \ref{pro:state_regularity_smooth}, with $r = 0$, to obtain
\[
 \|z - \Pi_{\T} z \|_s \lesssim h^{\vartheta}_{\T} | z |_{H^{s+\vartheta}(\Omega)} \lesssim h^{\vartheta}_{\T} \| w -  G_{\T} w \|_{L^2(\Omega)},
\]
where $\vartheta = \min \{ s, 1/2 - \epsilon \}$. The estimate \eqref{eq:error_estimate_for_G} allows us to conclude.
\end{proof}

 \subsection{A fully discrete scheme}
 \label{sec:fully_scheme}
 
We now design a fully discrete scheme to solve the state equation \eqref{eq:weak_formulation}. The discretization in time uses the backward Euler scheme. The space discretization hinges on the finite element space introduced in \S\ref{sec:space_discretization}.

Set $\zsf = 0$. The scheme computes the sequence $U_{\T}^\tau  \subset \V(\T)$, an approximation of the solution to problem \eqref{eq:weak_formulation} at each time step. We initialize the scheme by setting
 \begin{equation}
 \label{eq:initial_data_discrete}
  U_{\T}^{0} =  P_{\T}  \usf_0,
 \end{equation}
 where $P_{\T}$ denotes the $L^2(\Omega)$-orthogonal projection onto $\V(\T)$. For $k=0,\dots,\K-1$, $U_{\T}^{k+1} \in \V(\T)$ solves
\begin{equation}
\label{eq:fully_scheme}
U_{\T}^{k+1} \in \V(\T):
(\mathfrak{d} U_{\T}^{k+1} , W )_{L^2(\Omega)}  + \mathcal{A}(U_{\T}^{k+1}, W) = \langle \fsf^{k+1}, W \rangle
\quad
\forall \EO{W} \in \V(\T),
\end{equation}
where $\fsf^{k+1}= \tau^{-1} \int_{t_k}^{t_{k+1}} \fsf \diff t$. We recall that $\mathfrak{d}$ is defined by \eqref{1_discrete}.

The fully discrete scheme \eqref{eq:initial_data_discrete}--\eqref{eq:fully_scheme} is unconditionally stable.

\begin{theorem}[unconditional stability]
 \label{thm:stab}
 Let $U_{\T}^\tau$ be the solution to the fully discrete scheme \eqref{eq:initial_data_discrete}--\eqref{eq:fully_scheme}. If $\fsf \in L^2(0,T; H^{-s}(\Omega))$ and $\usf_0 \in L^2(\Omega)$, then
 \begin{equation}
\| U_{\T}^\tau \|_{\ell^{\infty}(L^2(\Omega))}^2 + | U_{\T}^\tau |_{\ell^2(H^s(\R^n))}^2
\lesssim \| \usf_0 \|_{L^2(\Omega)}^2 +  \| \fsf^{\tau} \|^2_{\ell^2(H^{-s}(\Omega))}.
\label{eq:stab}
 \end{equation}
The hidden constant is independent of the data, the solution $U_{\T}^\tau$, and the discretization parameters.
 \end{theorem}
\begin{proof}
Set $W = 2 \tau U_{\T}^{k+1}$ in \eqref{eq:fully_scheme}. The relation $2(a-b)a = a^2 - b^2 + (a-b)^2$ and Young's inequality yield
\begin{equation*}
\| U_{\T}^{k+1} \|_{L^2(\Omega)}^2 - \| U_{\T}^{k} \|_{L^2(\Omega)}^2 + \| U_{\T}^{k+1} - U_{\T}^{k} \|_{L^2(\Omega)}^2 + \tau \| U_{\T}^{k+1} \|_{s}^2 \lesssim \tau \| \fsf^{k+1} \|_{H^{-s}(\Omega)}^2.
\end{equation*}
The stability estimate \eqref{eq:stab} follows from adding the previous inequality over $k$.
\end{proof}

\subsection{$L^2(Q)$-error estimate}
\EO{As a technical instrument}, we introduce a semidiscrete approximation of problem \eqref{eq:weak_formulation}: Set $\zsf = 0$ and $U^0 = \usf_0$. For $k=0,\dots,\K-1$, $U^{k+1} \in \tilde H^s(\Omega)$ solves
\begin{equation}
\label{eq:semi_discrete}
(\mathfrak{d} U^{k+1} , \phi )_{L^2(\Omega)}  + \mathcal{A}(U^{k+1},\phi) = \langle \fsf^{k+1}, \phi \rangle
\quad
\forall \phi \in \tilde H^s(\Omega).
\end{equation}

The scheme \eqref{eq:semi_discrete} is unconditionally stable.

\begin{theorem}[unconditional stability]
 \label{thm:stab_improved}
Let $U^{\tau}$ be the solution to \eqref{eq:semi_discrete}. If $\fsf \in L^2(Q)$ and $\usf_0 \in \tilde H^s(\Omega)$, then 
 \begin{equation}
\| \mathfrak{d} U^\tau \|_{\ell^{2}(L^2(\Omega))}^2 + | U^\tau |_{\ell^{\infty}(H^s(\R^n))}^2
\lesssim | \usf_0 |_{H^s(\mathbb{R}^n)}^2 +  \| \fsf^{\tau} \|^2_{\ell^2(L^2(\Omega))},
\label{eq:stab_improved}
 \end{equation}
where the hidden constant is independent of the data, the solution $U^\tau$, and $\tau$.
 \end{theorem}
 \begin{proof}
 Set $W = U^{k+1} - U^k$ in \eqref{eq:fully_scheme}, use the relation $2(a-b)a = a^2 - b^2 + (a-b)^2$, and add over $k$.
 \end{proof}
 
 Define the piecewise linear function $\hat{U} \in C^{0,1}([0,T]; \tilde H^s(\Omega))$ by 
 \begin{equation}
 \label{hatV}
 \hat{U}(0) = U^0,
 \quad
 \hat{U}(t) = U^{k} + (t-t_k) \mathfrak{d} U^{k+1}, \quad t \in (t_k,t_{k+1}],
 \end{equation}
 for $k = 0,\dots,\K-1$. An important observation is that, for $t \in (t_k,t_{k+1}]$, $\partial_t \hat{U}(t) = \mathfrak{d} U^{k+1}$. We can thus rewrite the semidiscrete scheme \eqref{eq:semi_discrete}, for a.e.~$t \in (0,T)$, as
 \begin{equation}
 \label{eq:newsemi_discrete}
( \partial_t \hat{U}(t), \phi )_{L^2(\Omega)}  + \mathcal{A}(U^\tau(t),\phi) = \left\langle \fsf^{\tau}(t), \phi   \right\rangle \quad \forall \phi \in \tilde H^s(\Omega).
 \end{equation} 

Define $\hat{e}:= \usf - \hat{U}$ and $\bar e:= \usf - U^{\tau}$. We observe that $\hat e (0) = \bar e(0) = 0$. In addition, since the form $\mathcal{A}$ is bilinear and continuous, basic computations reveal that
\[
 \frac{\diff}{\diff t} \mathcal{A} \left( \int_0^t \bar{e}(\xi) \diff \xi, \int_0^t \bar{e}(\xi) \diff \xi\right)
 =
 2 \mathcal{A}\left( \int_0^t \bar{e}(\xi) \diff \xi, \bar{e}(t) \right).
\]
Consequently,
\begin{equation}
\label{eq:integral_positive}
   \int_0^T \mathcal{A} \left( \int_0^t \bar{e}(\xi) \diff \xi, \bar{e}(t) \right) \diff t  = \frac{1}{2}  \mathcal{A} \left( \int_0^T \bar{e}(t) \diff t,\int_0^T \bar{e}(t) \diff t \right) \geq 0.
 \end{equation}
 
We now derive an error estimate for the semidiscrete scheme \eqref{eq:semi_discrete}. 
 \begin{theorem}[semi-discrete error estimate]
 \label{thm:rate_state_1}
Let $\usf$ and $U^\tau$ be the solutions to \eqref{eq:weak_formulation} and \eqref{eq:semi_discrete}, respectively. If $\usf_0 \in \tilde H^s(\Omega)$ and $\fsf \in L^{\infty}(0,T;L^2(\Omega))$, then
 \begin{equation}
 \label{eq:semi_discrete_error_estimate}
   \| \usf - U^\tau \|_{L^2(0,T;L^2(\Omega))}  \lesssim \tau \left( | \usf_0 |_{H^s(\mathbb{R}^n)} +  \| \fsf \|_{L^{\infty}(0,T;L^2(\Omega))}\right).
 \end{equation}
 The hidden constant is independent of the data, $\usf$, $U^{\tau}$, and $\tau$.
 \end{theorem}
 \begin{proof}
\EO{Recall that }$\zsf = 0$ in \eqref{eq:weak_formulation}. Subtract from it \eqref{eq:newsemi_discrete} and integrate the resulting expression with respect to time. This yields, \EO{for a.e.~$t$,
 \begin{equation*}
 (  \bar{e}(t), \phi )_{L^2(\Omega)}  + \mathcal{A}\left(\int_0^t \bar{e}(\xi) \diff \xi, \phi \right) 
 = 
 ( \bar{e}(t) - \hat{e}(t) ,  \phi )_{L^2(\Omega)} 
 +
   \left\langle \int_0^t (\fsf(\xi) - \fsf^{\tau}(\xi)) \diff \xi,  \phi   \right\rangle
   %\quad \forall \phi \in \tilde H^s(\Omega), \quad \text{a.e. } t\in (0,T).
 \end{equation*}
for all $\phi \in \tilde H^s(\Omega)$.} Set, for a.e.~$t\in (0,T)$, $\phi = \bar{e}(t) \in \tilde H^s(\Omega)$. Integrate with respect to time, again,
and  invoke the identity \eqref{eq:integral_positive}, to arrive at
 \begin{multline}
 \label{eq:aux_estimate_1}
   \int_0^T \| \bar{e}(t) \|^2_{L^2(\Omega)}\diff t 
   \leq \left| \int_0^T \left \langle \int_0^t [\fsf(\xi) - \fsf^{\tau}(\xi)] \diff \xi, \bar{e}(t)   \right\rangle \diff t \right| \\
   + \left| \int_0^T (  \bar{e}(t) - \hat{e}(t) , \bar{e}(t) )_{L^2(\Omega)} \diff t \right| =: \textrm{I} + \textrm{II}. 
 \end{multline}

It thus suffices to estimate $\textrm{I}$ and $\textrm{II}$. To control the term $\textrm{I}$, we first notice that, since $\fsf^{k+1} = \tau^{-1}\int_{t_k}^{t_{k+1}} \fsf(t) \diff t$, we have, for $\ell \in \{1,\cdots,\mathcal{K}\}$,
 \[
   \int_0^{t_\ell} (\fsf(\xi) - \fsf^\tau(\xi) ) \diff \xi = \sum_{k=1}^{\ell} \int_{t_{k-1}}^{t_k} \left( \fsf(\xi) - \fsf^k \right) \diff \xi = 0.
 \]
\EO{If $t\in (t_\ell, t_{\ell+1})$, we have
$
   \int_0^t [\fsf(\xi) - \fsf^\tau(\xi) ] \diff \xi = \int_{t_\ell}^{t} [\fsf(\xi) - \fsf^\tau(\xi) ] \diff \xi \lesssim \tau \| \fsf \|_{L^\infty(0,T)}.
$
Thus,}
\[
\textrm{I}
\leq \int_0^T \left \|  \int_0^t [\fsf(\xi) - \fsf^{\tau}(\xi)] \diff \xi \right\|_{L^2(\Omega)} \| \bar{e}(t) \|_{L^2(\Omega)} \diff t
 \lesssim
   \tau^2 \| \fsf \|_{L^\infty(0,T,L^2(\Omega))}^2 + \frac14 \| \bar{e} \|_{L^2(Q)}^2.
\]

We now focus on estimating the term \textrm{II}. Since, on $(t_k,t_{k+1}]$, we have that $|\bar{e}(t) - \hat{e}(t)| \lesssim \tau |\mathfrak{d} U^{k+1}|$, we invoke the stability estimate \eqref{eq:stab_improved} to conclude that
 \[
   \int_0^T \| \hat{e}(t) -\bar{e}(t) \|^2_{L^2(\Omega)}\diff t \lesssim
   \tau^2 \left\| \mathfrak{d} U^{\tau} \right\|_{\ell^2(L^2(\Omega))}^2
   \lesssim \tau^2 \left( | \usf_0|_{H^s(\mathbb{R}^n)}^2 + \| \fsf^{\tau}\|^2_{\ell^2(L^2(\Omega))} \right),
 \]
which implies the bound
%following bound for the term $\mathrm{II}$:
$
   \mathrm{II}
   \leq \frac14 \| \bar{e} \|_{L^2(0,T;L^2(\Omega))}^2 
   + C \tau^2 \left( | \usf_0|_{H^s(\mathbb{R}^n)}^2 + \| \fsf^{\tau}\|^2_{\ell^2(L^2(\Omega))} \right).
$
 
 The desired estimate \eqref{eq:semi_discrete_error_estimate} follows from replacing the estimates for $\mathrm{I}$ and $\mathrm{II}$ into \eqref{eq:aux_estimate_1}. This concludes the proof.
\end{proof}

We now control the difference between the fully and the semidiscrete problems. 

\begin{theorem}[auxiliary error estimate]
\label{thm:rate_state_2}
\EO{Let $s \in (0,1)$ and $\Omega$ be a domain such that $\partial \Omega \in C^{\infty}$. Let $U^\tau$ and $U^{\tau}_{\T}$ be the solutions to problems \eqref{eq:semi_discrete} and \eqref{eq:fully_scheme}, respectively. If $\usf_0 \in \tilde H^{\vartheta+1/2-\epsilon}(\Omega)$ and $\fsf \in L^{2}(0,T;H^{1/2-s-\epsilon}(\Omega))$, then
\begin{equation*}
\label{estimate_aux}
\| U^\tau - U^{\tau}_{\T} \|_{\ell^2(L^2(\Omega))}  \lesssim h_{\T}^{\vartheta + 1/2 - \epsilon} \left( \|\usf_0 \|_{H^{\vartheta+1/2-\epsilon}(\mathbb{R}^n)} + \| \fsf \|_{L^{2}(0,T;H^{1/2-s-\epsilon}(\Omega))} \right),
\end{equation*}
where $\vartheta = \min \{s,1/2-\epsilon\}$ and $\epsilon > 0$ is arbitrarily small. The hidden constant does not depend on $U^{\tau}$, $U^{\tau}_{\T}$, or the problem data.}
\end{theorem}
\begin{proof}
\EO{As it is customary, we split the error as follows:
%into the so-called interpolation and approximation errors: 
\begin{equation}
\label{En}
E^\tau = (U^\tau - G_{\T} U^\tau) + (G_{\T} U^\tau - U_{\T}^\tau) =: \theta^\tau + \rho_{\T}^\tau;
\end{equation}
$G_{\T}$ denotes the elliptic projector defined in \eqref{eq:elliptic_projection}. To estimate $\theta^\tau$ we invoke estimate \eqref{eq:L2_error_estimate_for_G} with $\kappa = s + 1/2 - \epsilon$ and the regularity results of Theorem \ref{th:regularity_estimate} with $r = 1/2 - s - \epsilon$:
\begin{equation}
\| \theta^{\tau} \|_{\ell^2(L^2(\Omega))} \lesssim h_{\T}^{\vartheta + 1/2 -\epsilon } \| U^{\tau} \|_{\ell^2(H^{s+1/2 - \epsilon}(\Omega))}
\\ \lesssim h_{\T}^{\vartheta + 1/2 - \epsilon } \mathfrak{B}(\usf_0,\fsf),
\label{eq:estimate_for_theta}
\end{equation}
where $\mathfrak{B}(\usf_0,\fsf) := \| \usf_0 \|_{H^{1/2-\epsilon}(\mathbb{R}^n)} + \| \fsf \|_{L^{2}(0,T;H^{1/2 - s -\epsilon}(\Omega))}$ and $\vartheta = \min \{s, 1/2- \epsilon \}$. 

The estimate of the term $\rho^\tau_{\T}$ follows along the same lines of \cite[Lemma 3.8]{BGRV}. Let $W \in \V(\T)$. Set $\phi = W$ in \eqref{eq:semi_discrete}, multiply by $\tau$, sum from $k=0$ to $k=\ell$, and invoke the definition of the elliptic projection $G_{\T}$, given in \eqref{eq:elliptic_projection}, to obtain
 \begin{equation*}
 \langle   G_{\T} U^{\ell+1}, W \rangle  + \mathcal{A} \left ( \sum_{k=0}^{\ell} \tau G_{\T} U^{k+1},W \right) =  \left( \sum_{k=0}^{\ell} \tau  \fsf^{k+1}, W \right) 
 %- \langle \theta^{\ell+1}, W \rangle
 + \langle \usf_0, W \rangle -   \langle  \theta^{\ell +1} , W \rangle.
 %+  \langle  G_{\T} U^{\ell+1} - U^{\ell+1} , W \rangle.
% + \mathcal{A} \left(\sum_{k=0}^{\ell} \tau  \usf(t_{k+1}) - \int_0^{t_{l+1}}  \usf \diff t ,W \right).
  \label{eq:tl+1}
 \end{equation*}
Similarly,
\begin{equation*}
 \label{eq:tl+1_discrete}
 \langle U_{\T}^{\ell+1}, W \rangle  + \mathcal{A} \left(\sum_{k=0}^{\ell} \tau U_{\T}^{k+1} , W \right) =  \langle  P_{\T} \usf_0, W \rangle + \left(\sum_{k=0}^{\ell} \tau \fsf^{k+1} , W \right).
 \end{equation*}
Consequently, for $\ell = 0,\dots,\K-1$, $\rho_{\T}^{\ell+1} \in \V(\T)$ solves
 \begin{equation}
 \langle   \rho_{\T}^{\ell+1}, W \rangle  + \mathcal{A} \left ( \sum_{k=0}^{\ell} \tau \rho_{\T}^{k+1},W \right) =  \langle  \usf_0 - P_{\T} \usf_0, W \rangle 
 - \langle  \theta^{\ell +1} , W \rangle \quad \forall W \in \V(\T).
  \label{eq:rho_fully_discrete}
 \end{equation}
Observe that $\rho_{\T}^{0}  = G_{\T} \usf_0 - P_{\T} \usf_0$. Set $W = \rho_{\T}^{\ell + 1}$ in \eqref{eq:rho_fully_discrete} to arrive at
\begin{multline*}
 \| \rho_{\T}^{\ell+1} \|^2_{L^2(\Omega)} + \tau   \mathcal{A} \left ( \sum_{k=0}^{\ell} \rho_{\T}^{k+1}, \rho_{\T}^{\ell+1} \right) \leq \frac{1}{4}  \| \rho_{\T}^{\ell+1} \|^2_{L^2(\Omega)}  + C_1 \|P_{\T} \usf_0 - \usf_0 \|^2_{L^2(\Omega)}
 \\
 + \frac{1}{4} \| \rho_{\T}^{\ell+1} \|^2_{L^2(\Omega)} + C_2  \| \theta^{\ell+1} \|^2_{L^2(\Omega)},
 \end{multline*}
where $C_1, C_2 >0$. Multiply the previous inequality by $\tau$ and add over $\ell$ to obtain
 \begin{equation*}
 \| \rho_{\T}^{\tau} \|^2_{\ell^2(L^2(\Omega))} + \tau^2\sum_{\ell = 0}^{\K-1} \mathcal{A} \left ( \sum_{k=0}^{\ell} \rho_{\T}^{k+1}, \rho_{\T}^{\ell+1} \right) 
 \lesssim  \|P_{\T} \usf_0 - \usf_0 \|^2_{L^2(\Omega)} + \| \theta^{\tau} \|^2_{\ell^2(L^2(\Omega))}.
 \end{equation*}
We now invoke \cite[inequality (3.40)]{BGRV}, which reads 
\[
\tau^2\sum_{\ell = 0}^{\K-1} \mathcal{A} \left ( \sum_{k=0}^{\ell} \rho_{\T}^{k+1}, \rho_{\T}^{\ell+1} \right)  \geq C \left\| \tau\sum_{k=0}^{\K-1} \rho_{\T}^{k+1} \right \|_{\tilde H^s(\Omega)}^2,
\] 
with $C>0$, and the assumption $\usf_0 \in \tilde{H}^{\vartheta+1/2-\epsilon}(\Omega)$, to obtain
\begin{equation*}
 \| \rho_{\T}^{\tau} \|_{\ell^2(L^2(\Omega))} 
 \lesssim h_{\T}^{\vartheta+1/2-\epsilon} \| \usf_0 \|_{H^{\vartheta+1/2-\epsilon}(\R^n)} + \| \theta^{\tau} \|_{\ell^2(L^2(\Omega))} .
% \\
% & \lesssim  h_{\T}^{s+1/2-\epsilon} \left( \|\usf_0 \|_{H^{s+1/2-\epsilon}(\mathbb{R}^n)} + \| \fsf \|_{L^{2}(0,T;H^{1/2-s-\epsilon}(\Omega))} \right).
\end{equation*}
Invoke \eqref{eq:estimate_for_theta} to finalize the estimate for $ \| \rho_{\T}^{\tau} \|_{\ell^2(L^2(\Omega))}$. This concludes the proof.}
\end{proof}
 
We collect the estimates of Theorems \ref{thm:rate_state_1} and \ref{thm:rate_state_2} to derive a $L^2(Q)$-error estimate for the fully discrete scheme \eqref{eq:initial_data_discrete}--\eqref{eq:fully_scheme}. \EO{To simply the presentation, we define
\begin{equation}
\EO{\mathfrak{A} (\usf_0,\fsf):= \| \usf_0 \|_{H^{\vartheta+1/2-\epsilon}(\mathbb{R}^n)} + \| \fsf \|_{L^{2}(0,T;H^{1/2-s-\epsilon}(\Omega))} + \| \fsf \|_{L^{\infty}(0,T;L^{2}(\Omega))}.}
\label{eq:frakA}
\end{equation}

\begin{theorem}[error estimate for fully discrete scheme]
\label{thm:rate_state_final}
Let $s \in (0,1)$ and $\Omega$ be a domain such that $\partial \Omega \in C^{\infty}$.
Let $\usf$ and $U^{\tau}_{\T}$ solve \eqref{eq:weak_formulation} and \eqref{eq:fully_scheme}, respectively. If $\usf_0 \in \tilde H^{\vartheta+1/2-\epsilon}(\Omega)$ and $\fsf \in L^{2}(0,T;H^{1/2-s-\epsilon}(\Omega)) \cap L^{\infty}(0,T;L^2(\Omega))$, then
 \begin{equation*}
 \label{eq:fully_scheme_error_estimate}
   \| \usf - U^{\tau}_{\T} \|_{L^2(Q) }  \lesssim  (\tau + h_{\T}^{\vartheta + 1/2 - \epsilon})
\mathfrak{A} (\usf_0,\fsf),
 \end{equation*}
where  $\vartheta = \min \{s,1/2-\epsilon\}$ and $\epsilon > 0$ is arbitrarily small. 
The hidden constant does not depend on $h_{\T}$, $\tau$, $\usf$, $U_{\T}^{\tau}$, or the problem data.
\end{theorem}
}
\begin{remark}[error estimate]\rm
\EO{If $s < \tfrac{1}{2}$, the estimate of Theorem \ref{thm:rate_state_final} reads as
\[ 
\| \usf - U^{\tau}_{\T} \|_{L^2(Q) } \lesssim \tau + h_{\T}^{s+1/2-\epsilon}.
\]
This error estimate is in agreement with respect to regularity; see Theorem \ref{th:regularity_estimate}. In contrast, when $s \geq \tfrac{1}{2}$ the error estimate is suboptimal: $\| \usf - U^{\tau}_{\T} \|_{L^2(Q) } \lesssim \tau + h_{\T}^{1-\epsilon}$; the responsible being the duality argument employed in the proof of Proposition \ref{pro:L2_elliptic}.}
\end{remark}
 
\section{Approximation of parabolic fractional control problem}
\label{sec:approximation_control}

In this section, we introduce an implicit fully-discrete scheme to approximate the solution of the fractional optimal control problem \eqref{eq:min}--\eqref{eq:cc}. The scheme discretizes the control variable with piecewise constant functions. The state variable is discretized with standard piecewise linear finite elements in space, as detailed in section \ref{sec:space_discretization}, and with the backward Euler scheme in time, as described in section \ref{sec:time_discretization}.

To simplify the exposition, in what follows we assume that $\asf$ and $\bsf$ are constants.

\subsection{An implicit fully discrete-scheme}
\label{sub:fd_control}
To discretize the control variable, we introduce the finite element space of piecewise constant functions over $\T$,
\[
\mathbb{Z}(\T) = \left\{ Z \in L^\infty(\Omega): Z|_K \in \mathbb{P}_0(K) \ \forall K \in \T \right\},
\]
and the space of piecewise constant functions in time and space,
\begin{equation}
\mathbb{Z}(\mathcal{T},\T) = \left\{ Z^\tau \subset L^\infty(Q) : Z^k \in \mathbb{Z}(\T) \right\}.
\end{equation}
The space of discrete admissible controls is defined as 
$
 \Zad(\mathcal{T},\T) = \Zad \cap \mathbb{Z}(\mathcal{T},\T),
$
where $\Zad$ is defined in \eqref{ac}. 

To perform an a priori error analysis, it is useful to introduce the $L^2(Q)$-orthogonal projection onto $\mathbb{Z}(\calT,\T)$. This operator, $\Pi_{\T}^{\mathcal{T}}: L^2(Q) \rightarrow \mathbb{Z}(\mathcal{T},\T)$, is defined by
\begin{equation}
\label{eq:or_pro}
\zsf \in L^2(Q): \quad  (\zsf  -  \Pi^{\mathcal{T}}_{\T}  \zsf , Z)_{L^2(Q)} = 0 \qquad \forall Z \in \mathbb{Z}(\mathcal{T},\T).
\end{equation}
If $\zsf  \in H^1(0,T;L^2(\Omega)) \cap L^2(0,T;H^{\kappa}(\Omega))$, \EO{with $\kappa \in (0,1]$,} we have the error estimate
\begin{equation}
 \label{eq:or_pro_prop}
 \| \zsf  - \Pi^{\mathcal{T}}_{\T}\zsf  \|_{L^2(Q)} \lesssim h_{\T}^{\kappa} \| \zsf  \|_{L^2(0,T; H^{\kappa}(\Omega))} + \tau \| \partial_t \zsf  \|_{L^2(Q)}.
\end{equation}
An important observation in favor of $\Pi_{\T}^{\mathcal{T}}$ is that $\Pi_{\T}^\calT \Zad \subset \Zad(\calT, \T)$. We recall that $\asf$ and $\bsf$, that define the set \eqref{ac}, are constant.

We define the discrete functional \EO{$J_{\T}^{\calT} : \V(\T)^\K \times \mathbb{Z}(\mathcal{T}, \T) \to \R$} by
\[
J_{\T}^{\calT}(U^\tau_{\T},Z_{\T}^{\tau}) = \frac{1}{2} \| U_{\T}^{\tau} - \usf_d^{\tau}\|^2_{\ell^2(L^2(\Omega))} +  \frac{\mu}{2}\| Z^{\tau}_{\T} \|^2_{\ell^2(L^2(\Omega))}.
\]
Notice that, if $\usf_d^\tau = \usf_d$ we would have that $J_{\T}^{\calT}(w,r) = J(w,r)$ whenever $w^\tau = w$ and $r^\tau = r$; see Remark \ref{rk:indentification}.

With this notation at hand, we introduce the following fully discrete scheme for our parabolic fractional optimal control problem: \EO{Find
$
\min J_{\T}^{\calT}(U_{\T}^\tau,Z_{\T}^\tau)
$
subject} to the discrete equation: initialize as in \eqref{eq:initial_data_discrete} and, for $k=0,\dots,\K-1$,
\begin{equation}
\label{fully_state}
\EO{U_{\T}^{k+1} \in \V(\T):}
\quad
( \mathfrak{d}  U_{\T}^{k+1} , W )_{L^2(\Omega)}  + \mathcal{A}(U_{\T}^{k+1},W) = \left\langle \fsf^{k+1} + Z_{\T}^{k+1}, W   \right\rangle,
\end{equation}
for all $W \in \V(\T)$, and the control constraints $Z^{\tau}_{\T} \in \Zad(\calT, \T)$. 

\subsection{First order optimality conditions}
We provide first order necessary and sufficient optimality conditions. To accomplish this task, we define the discrete adjoint problem: Find $P_{\T}^{\tau} \subset \V(\T)$ such that $P_{\T}^{\K} = 0$ and, for $k = \K-1, \ldots, 0$,
\begin{equation}
\label{fully_adjoint}
\EO{P_{\T}^k \in \V(\T):}
\quad
( \bar{\mathfrak{d}} P_{\T}^{k} , W )_{L^2(\Omega)}  + \mathcal{A}(P_{\T}^k,W) = \langle U_{\T}^{k+1} - \usf_d^{k+1}, W \rangle
\end{equation} 
for all $W \in \V(\T)$. The difference operator $\bar{\mathfrak{d}}$ is defined in \eqref{1_discrete_forward}.

The optimality condition reads: $( \bar{U}^{\tau}_{\T}, \bar{Z}^{\tau}_{\T})$ is optimal for the scheme of section \ref{sub:fd_control} if and only if $\bar{U}^{0}_{\T} = P_{\T} \usf_0$, for $k = 0, \cdots, \mathcal{K}-1$, $\bar U_{\T}^{k+1} \in \V(\T)$ solves \eqref{fully_state}, and 
 \begin{equation}
 \label{eq:op_discrete}
  ( \mu \bar{Z}^{\tau}_{\T} + \bar{P}_{\T}^{\tau}, Z -\bar{Z}^{\tau}_{\T})_{L^2(Q)} \geq 0 \quad 
  \forall Z \in \Zad(\mathcal{T},\T),
 \end{equation}
 where $\bar{P}^{\tau}_{\T}$ solves \eqref{fully_adjoint}. Set $Z^\tau = Z \chi_{(t_{k-1},t_k]}$ with $Z \in \mathbb{Z}(\T_\Omega)$ and $\asf \leq Z \leq \bsf$ in \eqref{eq:op_discrete}. We thus obtain that \eqref{eq:op_discrete} can be equivalently written as
 \[
   ( \bar P_{\T}^k + \mu \bar{Z}^k_{\T}, Z - \bar{Z}^k_{\T} )_{L^2(\Omega)} \geq0 \quad \forall Z \in \mathbb{Z}(\T), \quad \asf \leq Z \leq \bsf,
   \quad \forall k = 1,\ldots,\K.
 \]
 
\subsection{Auxiliary problems}
\label{sub:auxiliary_problemsl}

We introduce two auxiliary problems that will be instrumental to derive error estimates for the fully discrete scheme of section \ref{sub:fd_control}.

The first problem reads as follows: Find $Q_{\T}^{\tau} \subset \V(\T)$ such that $Q_{\T}^{\K} = 0$ and, for $k = \K-1, \ldots, 0$, $Q_{\T}^k \in \V(\T)$ solves
 \begin{equation}
 \label{adjoint_aux1}
  ( \bar{\mathfrak{d}} Q_{\T}^{k} , W )_{L^2(\Omega)}  + \mathcal{A}(Q_{\T}^{k},W) = \left\langle  \bar{\usf}^{k+1} - \usf_d^{k+1}, W   \right\rangle
 \end{equation}
 for all $W \in \V(\T)$; $\bar{\usf} = \bar{\usf}(\bar \zsf)$ denotes the solution to \eqref{eq:weak_formulation} with $\zsf$ replaced by $\ozsf$.
 
The second auxiliary problem is:  Find $R_{\T}^{\tau} \subset \V(\T)$ such that $R_{\T}^{\K} = 0$ and, for $k = \K-1, \ldots, 0$, $R_{\T}^k \in \V(\T)$ solves
 \begin{equation}
 \label{adjoint_aux2}
  ( \bar{\mathfrak{d}} R_{\T}^{k} , W )_{L^2(\Omega)}  + \mathcal{A}(R_{\T}^{k},W) = \langle  U^{k+1}_{\T}(\bar \zsf) - \usf_d^{k+1}, W   \rangle
 \end{equation}
 for all $W \in \V(\T)$; $U^{k+1}_{\T}(\bar \zsf)$ denotes the solution to \eqref{fully_state} with $Z_{\T}^{k+1}$ replaced by $\bar \zsf^{k+1}$.
 
\subsection{A priori error analysis: $s \in (0,1)$}
\label{sub:apriori_control}
 
\EO{We derive an a priori error estimate for the error approximation of the control variable.}
 
 \begin{theorem}[error estimate for control approximation]
 \label{thm:control_error}
\EO{Let $s \in (0,1)$ and $\Omega$ be a domain such that $\partial \Omega \in C^{\infty}$. Let $\bar{\zsf}$ be the optimal control for \eqref{eq:min}--\eqref{eq:cc} and let $\bar{Z}^{\tau}_{\T}$ be the optimal control for the scheme of section \ref{sub:fd_control}. If $\usf_0 \in \tilde H^{\beta}(\Omega)$ and $\fsf, \usf_d \in L^{\infty}(0,T;L^2(\Omega)) \cap L^2(0,T;H^{\beta}(\Omega))$, for every $\beta<1$, then
 \begin{equation}
 \| \bar{\zsf} - \bar{Z}^{\tau}_{\T}  \|_{L^2(Q)} \lesssim \tau + h_{\T}^{\gamma},
 \label{eq:final_error_estimate_control}
 \end{equation}
where $\gamma = \min \{1,s + 1/2 - \epsilon \}$ and $\epsilon > 0$ is arbitrarily small. The hidden constant is independent of the optimal continuous and discrete variables and the discretization parameters, but depends on the problem data.}
\end{theorem}
\begin{proof} 
We proceed in several steps.

\noindent \emph{Step 1.} Set $\zsf = \bar{Z}^{\tau}_{\T}$ in \eqref{viz} and $Z = \Pi^{\mathcal{T}}_{\T} \ozsf$ in \eqref{eq:op_discrete}, where $\Pi^{\mathcal{T}}_{\T}$ denotes the $L^2(Q)$-orthogonal projection onto $\mathbb{Z}(\calT,\T)$. Add the obtained inequalities to arrive at
 \[
  \mu \| \bar{\zsf} - \bar{Z}^{\tau}_{\T}  \|_{L^2(Q)}^2 \leq ( \bar{\psf} - \bar{P}^{\tau}_{\T} ,\bar{Z}^{\tau}_{\T} - \ozsf )_{L^2(Q)} + ( \bar{P}^{\tau}_{\T} + \mu \bar{Z}^{\tau}_{\T},\Pi^{\mathcal{T}}_{\T} \ozsf - \ozsf )_{L^2(Q)}.
 \]
 We recall that the adjoint state $\bar{\psf}$ solves \eqref{fractional_adjoint} with $\usf$ replaced by $\bar \usf$ 
 and its fully discrete counterpart $\bar{P}^{\tau}_{\T}$ is defined as the solution to \eqref{fully_adjoint} with $U_{\T}^{k+1}$ replaced by $\bar U_{\T}^{k+1}$.
\\

 \noindent \emph{Step 2.} We invoke the solutions to the auxiliary problems \eqref{adjoint_aux1} and \eqref{adjoint_aux2} to write $\bar{\psf} - \bar{P}^{\tau}_{\T} = (\bar{\psf} -  Q^{\tau}_{\T}) + (Q^{\tau}_{\T} - R^{\tau}_{\T}) + (R^{\tau}_{\T} - \bar P^{\tau}_{\T})$. Since $Q^{\tau}_{\T}$ solves \eqref{adjoint_aux1}, the estimate for the term $\bar{\psf} -  Q^{\tau}_{\T}$ follows immediately from \EO{Theorem \ref{thm:rate_state_final}:
 \begin{equation}
 \label{eq:p-Q}
 \| \bar{\psf} - Q^{\tau}_{\T} \|_{L^2(Q)} \lesssim \left(\tau + h_{\T}^{\vartheta +1/2 - \epsilon} \right) \mathfrak{A}(0,\bar \usf-\usf_d),
 \end{equation}
where $\mathfrak{A}$ is defined in \eqref{eq:frakA} and $\vartheta = \min \{s, 1/2 - \epsilon \}$ with $\epsilon >0$ being arbitrarily small. In view of the energy estimate \eqref{eq:energy_u} and the regularity results of Theorem \ref{thm:regularity_space} we obtain that  $\bar \usf \in L^2(0,T;H^{s+1/2 - \epsilon}(\Omega)) \cap L^{\infty}(0,T;L^2(\Omega))$ for $\epsilon>0$ arbitrarily small. The assumption on $\usf_d$ thus yields $\mathfrak{A}(0,\bar \usf-\usf_d) < \infty$.}
%Invoke \eqref{eq:energy_u} to complete the previous error estimate:}
% \begin{equation}
%\| \bar{\psf} - Q^{\tau}_{\T} \|_{L^2(Q)} \lesssim \left(\tau + h_{\T}^{\vartheta +1/2 - \epsilon} \right) \left( \Sigma(\usf_0,\fsf + \bar \zsf) + \| \usf_d  \|_{L^{\infty}(0,T;L^2(\Omega))} 
%  \right),
%  \label{eq:p-Q}
% \end{equation}
% where $\Sigma$ is defined in \eqref{eq:Sigma}.
 \\
 
 \noindent \emph{Step 3.} The goal of this step is to control the difference $Q^{\tau}_{\T} - R^{\tau}_{\T}$. To accomplish this task, we first invoke the stability result of Theorem~\ref{thm:stab} and then the error estimate of Theorem \ref{thm:rate_state_final}. These arguments allow us to obtain
 \begin{equation}
 \| Q^{\tau}_{\T} - R^{\tau}_{\T}\|_{L^2(Q)} \lesssim 
 \| \bar{\usf} - U^{\tau}_{\T}(\ozsf) \|_{L^2(Q)}
 \lesssim \EO{\left(\tau + h_{\T}^{\vartheta +1/2 - \epsilon} \right) \mathfrak{A}(\usf_0, \fsf + \bar \zsf),}
% \lesssim (\tau + h^{2\gamma})  \Sigma(\usf_0, \fsf + \bar \zsf),
 \label{eq:Q-R}
 \end{equation}
\EO{where the hidden constant is independent of $h_{\T}$ and $\tau$.
The regularity results of Theorem \ref{thm:regularity_space} guarantee that $\bar \zsf \in H^{\gamma}(\Omega)$ where $\gamma = \min \{1,s+1/2-\epsilon \}$ and $\epsilon >0$ is arbitrarily small. Thus, $\bar \zsf \in L^2(0,T; H^{1/2-s-\epsilon}(\Omega)) \cap L^{\infty}(0,T;L^2(\Omega))$. In view of the assumptions on $\fsf$ and $\usf_0$ we can thus conclude that $\mathfrak{A}(\usf_0, \fsf + \bar \zsf) < \infty$.}
\\

  \noindent \emph{Step 4.} We handle the term $R^{\tau}_{\T} - \bar{P}^{\tau}_{\T}$ in view of an argument based on summation by parts.  First, we define
\[
\Psi^{k}: = \bar{P}_{\T}^k - R_{\T}^k, \quad \Phi^{k+1} := \bar{U}_{\T}^{k+1} - U_{\T}^{k+1}(\ozsf).
\]  
%  $\Psi^{k}: = \bar{P}_{\T}^k - \bar{R}_{\T}^k$ and $\Phi^{k+1} := \bar{U}_{\T}^{k+1} - U_{\T}^{k+1}(\ozsf)$. 
\EO{Set} $\Psi^{k}$ and  $\Phi^{k+1}$ in the problems that $\bar{U}_{\T}^{\tau} - U_{\T}^{\tau}(\ozsf)$ and $\bar{P}_{\T}^{\tau} - \bar{R}_{\T}^{\tau}$ solve, respectively. In view of the fact that $\psi^{\K} = 0 = \Phi^{0}$, invoke the discrete summation by parts formula
  \[
  \sum_{k=0}^{\K-1} \tau ( \mathfrak{d}\Phi^{k+1},\psi^k) = -  \sum_{k=0}^{\K-1} \tau (\Phi^{k+1}, \mathfrak{d}\psi^{k+1}) = \sum_{k=0}^{\K-1} \tau (\Phi^{k+1}, \bar{\mathfrak{d}}\psi^{k})
  \]
to conclude that \EO{$
   ( R^{\tau}_{\T} - \bar{P}^{\tau}_{\T}, \bar{Z}^{\tau}_{\T} - \ozsf )_{L^2(Q)} \leq 0$.}
 \\
 
  \noindent \emph{Step 5.} \EO{We now bound} $( \bar{P}^{\tau}_{\T} + \mu \bar{Z}^{\tau}_{\T},\Pi^{\mathcal{T}}_{\T} \ozsf - \ozsf )_{L^2(Q)}$. To accomplish this task, we write
\begin{multline}
  ( \bar{P}^{\tau}_{\T} + \mu \bar{Z}^{\tau}_{\T},\Pi^{\mathcal{T}}_{\T} \ozsf - \ozsf )_{L^2(Q)} 
  = 
  (\bar \psf + \mu \ozsf, \Pi^{\mathcal{T}}_{\T} \ozsf - \ozsf )_{L^2(Q)} 
  + 
  (\bar{P}_{\T}^\tau - Q_{\T}^\tau, \Pi^{\mathcal{T}}_{\T} \ozsf - \ozsf)_{L^2(Q)} 
  \\
  +  (Q_{\T}^\tau - \bar \psf, \Pi^{\mathcal{T}}_{\T} \ozsf - \ozsf)_{L^2(Q)} 
  +  \mu( \bar{Z}_{\T}^\tau - \ozsf,
  \Pi^{\mathcal{T}}_{\T} \ozsf - \ozsf)_{L^2(Q)} = \textrm{I} + \textrm{II} + \textrm{III} + \textrm{IV}.
\end{multline}
We recall that the auxiliary variable $Q_{\T}^\tau$ is defined as the solution to \eqref{adjoint_aux1}. 

To estimate the term $\textrm{I}$ we invoke property \eqref{eq:or_pro}, that defines $\Pi^{\mathcal{T}}_{\T}$, and the estimate \eqref{eq:or_pro_prop}. We can thus obtain, \EO{for $\gamma = \min \{1,s+1/2-\epsilon \}$ and $\epsilon >0$ sufficiently small,}
\begin{multline*}
\textrm{I} = (\bar \psf + \mu \ozsf - \Pi^{\mathcal{T}}_{\T}(\bar \psf + \mu \ozsf), \Pi^{\mathcal{T}}_{\T} \ozsf - \ozsf )_{L^2(Q)}  
\lesssim
\left(h_{\T}^{\gamma} \| \ozsf \|_{L^2(0,T;H^{\gamma}(\Omega))} + \tau \| \partial_t \ozsf  \|_{L^2(Q)}\right)
\\
\cdot \left( h_{\T}^{\gamma} \| \bar \psf + \mu \ozsf   \|_{L^2(0,T;H^{\gamma}(\Omega))} + \tau \| \partial_t (\bar \psf + \mu \ozsf )  \|_{L^2(Q)}\right).
\end{multline*}
Notice that $\|\opsf + \mu \ozsf \|_{L^2(0,T;H^\gamma(\Omega))}$ and $\| \partial_t(\opsf + \mu \ozsf)\|_{L^2(Q)}$ are uniformly controlled by the problem data; see the regularity estimates of Theorems \ref{th:regularity_time} and \ref{thm:regularity_space}.

In what follows we control $\mathrm{II}$. To accomplish this task, we first notice that
\[
 \| \bar{P}_{\T}^\tau - Q_{\T}^\tau \|_{L^2(Q)} \leq  \| \bar{P}_{\T}^\tau - R_{\T}^\tau \|_{L^2(Q)} +  \| R_{\T}^\tau - Q_{\T}^\tau \|_{L^2(Q)},
\]
where the auxiliary variable $R_{\T}^\tau$ is defined as the solution to \eqref{adjoint_aux2}.  The term $\| R_{\T}^\tau - Q_{\T}^\tau \|_{L^2(Q)}$ is bounded as in \eqref{eq:Q-R}. It thus suffices to bound $\| \bar{P}_{\T}^\tau - R_{\T}^\tau \|_{L^2(Q)}$. To do this, we invoke the stability estimate \eqref{eq:stab}, twice, to arrive at
\[
\| \bar{P}_{\T}^\tau - R_{\T}^\tau \|_{L^2(Q)} 
\lesssim 
\| \bar U_{\T}^{\tau} - U_{\T}^{\tau} (\bar \zsf) \|_{L^2(Q)} 
\lesssim 
\| \bar Z_{\T}^{\tau} - \bar \zsf \|_{L^2(Q)}.
\]
We thus obtain that\EO{ $\| \bar{P}_{\T}^\tau - Q_{\T}^\tau \|_{L^2(Q)} \lesssim  (\tau + h_{\T}^{\vartheta +1/2 - \epsilon} ) \mathfrak{A}(\usf_0, \fsf + \bar \zsf) + \| \bar Z_{\T}^{\tau} - \bar \zsf \|_{L^2(Q)}$.} We now invoke the Cauchy--Schwarz inequality,  the previous estimate for $\bar{P}_{\T}^\tau - Q_{\T}^\tau$, the error estimate \eqref{eq:or_pro_prop}, and Young's inequality to arrive at
\begin{multline*}
\mathrm{II} \leq \| \bar{P}_{\T}^\tau - Q_{\T}^\tau \|_{L^2(Q)} \| \Pi^{\mathcal{T}}_{\T} \ozsf - \ozsf  \|_{L^2(Q)} 
\leq
\frac{\mu}{4}\| \bar Z_{\T}^{\tau} - \bar \zsf \|_{L^2(Q)}^2 
+ 
C \left( \EO{\tau^2 \mathfrak{A}^2(\usf_0, \fsf + \bar \zsf)} \right.
\\
+ 
\left.
\EO{h_{\T}^{2(\theta+1/2-\epsilon)} \mathfrak{A}^2(\usf_0, \fsf + \bar \zsf)}
+ 
h_{\T}^{2\gamma} \|\ozsf \|^2_{L^2(0,T;H^{\gamma}(\Omega))} + \tau^2 \| \partial_t \ozsf  \|^2_{L^2(Q)} \right),
\end{multline*}
where $C>0$ and $\vartheta= \min \{s,1/2-\epsilon\}$ with $\epsilon >0$ arbitrarily small.
%\EO{We recall that $\gamma = \min \{1,s+1/2-\epsilon \}$ and $\vartheta= \min \{s,1/2-\epsilon\}$ with $\epsilon >0$ arbitrarily small.}

The control of the term $\mathrm{III}$ follows from \eqref{eq:p-Q} and \eqref{eq:or_pro_prop}. In fact, we have
\begin{equation*}
 \mathrm{III} \lesssim 
% (\tau + h^{2\gamma}) \left( \Sigma(\usf_0,\fsf + \bar \zsf) + \| \usf_d  \|_{L^{\infty}(0,T;L^2(\Omega))}  \right)
 \EO{\left(\tau + h_{\T}^{\vartheta +1/2 - \epsilon} \right) \mathfrak{A}(0,\bar \usf-\usf_d)
\left(h_{\T}^{\gamma} \| \ozsf \|_{L^2(0,T;H^{\gamma}(\Omega)} + \tau \| \partial_t \ozsf  \|_{L^2(Q)}\right).}
\end{equation*}

The term $\mathrm{IV}$ can be bounded in view of similar arguments.

\noindent \emph{Step 6.} The assertion follows from collecting all the estimates we obtained in previous steps. This concludes the proof.
\end{proof} 

\begin{remark}[error estimate]\rm
\EO{If $s>\tfrac12$, \eqref{eq:final_error_estimate_control} reads $ \| \bar{\zsf} - \bar{Z}^{\tau}_{\T}  \|_{L^2(Q)} \lesssim \tau + h_{\T}$. This estimate is \emph{optimal} with respect to approximation. When $s \leq \tfrac12$, it reads
\[
\| \bar{\zsf} - \bar{Z}^{\tau}_{\T}  \|_{L^2(Q)} \lesssim \tau + h_{\T}^{s+1/2-\epsilon},
\]
for $\epsilon>0$ arbitrarily small. This estimate is \emph{suboptimal} in terms of approximation but is in \emph{agreement} with the regularity results derived in Theorem \ref{thm:regularity_space} for $\bar \zsf$.}
\end{remark}

%%%%%%%%%%%%%%%%%%%%%%%%%%%%%%%%%%%%%%%%%%%%%%%%%%%%%%%%%%%%%%%%%%%%%%%%%%%%%%%%%%%%%%
\section{Numerical examples}
\label{sec:numerical-examples}

We present a series of numerical examples that illustrate the performance of the fully discrete scheme proposed in section \ref{sub:fd_control} when solving the optimal control problem \eqref{eq:min}--\eqref{eq:cc}. We consider one- and two-dimensional numerical experiments posed on the domain $B(0,1) \times (0,T)$, where $B(0,1)$ denotes the interval $(0,1)$, when $n=1$, and the circle of radius $1$ centered at $(0,0)$, when $n=2$.

\subsection{Exact solutions}
We let $n \in \{1,2\}$, $\Omega = B(0,1)$, and $s \in (0,1)$. We consider the fractional Poisson problem: Find $u$ such that
\begin{equation}
  \left(-\Delta\right)^{s}u = f \text{ in }\Omega, \quad
  u=0\text{ in }\Omega^{c}.
  \label{eq:fractional_Poisson}
\end{equation}

\EO{Let \(P_{k}^{(\alpha,\beta)}\) denote the Jacobi polynomials and \(x_{+}=\max\{0,x\}\). If \(n=1\) and $f$ is}
\begin{equation*}
  % f_{n,0}&= \frac{2^{2s}\Gamma\left(1+s+n\right)\Gamma\left(1/2+s+n\right)}{\Gamma\left(1+n\right)\Gamma\left(1/2+n\right)} P_{n}^{(s,-1/2)}\left(2r^{2}-1\right), &n\geq0
 % \\
  f_{k,0}^{1D}(x)= 2^{2s}\Gamma\left(1+s\right)^{2}\binom{s+k-1/2}{s}\binom{s+k}{s} P_{k}^{(s,-1/2)}\left(2x^{2}-1\right), \quad k\geq0,
\end{equation*}
then the solution $u$ is given \EO{by
$
  u_{k,0}^{1D}(x)= P_{k}^{(s,-1/2)}\left(2x^{2}-1\right) \left(1-x^{2}\right)^{s}_{+},
$
where}
%Here, \(P_{k}^{(\alpha,\beta)}\) denote the Jacobi polynomials, \(x_{+}=\max\{0,x\}\), and
\[
\binom{x}{y}=\frac{\Gamma\left(x+1\right)}{\Gamma\left(y+1\right)\Gamma\left(x-y+1\right)} 
\]
correspond to the generalized binomial coefficients.
When the right-hand side is
\begin{equation*}
  % f_{n,1}&= \frac{2^{2s}\Gamma\left(1+s+n\right)\Gamma\left(3/2+s+n\right)}{\Gamma\left(1+n\right)\Gamma\left(3/2+n\right)} x P_{n}^{(s,1/2)}\left(2r^{2}-1\right), &n\geq0
 % \\
  f_{k,1}^{1D}(x)= 2^{2s}\Gamma\left(1+s\right)^{2}\binom{s+k+1/2}{s}\binom{s+k}{s} x P_{k}^{(s,1/2)}\left(2x^{2}-1\right), \quad k\geq0,
\end{equation*}
\EO{then
$
  u_{k,1}^{1D}(x)= x P_{k}^{(s,1/2)}\left(2x^{2}-1\right) \left(1-x^{2}\right)^{s}_{+}.
$}

If \(n=2\) and the right-hand side, in polar coordinates, reads
\begin{align*}
  % f_{n,\ell}
  % &= \frac{2^{2s}\Gamma\left(1+s+n\right)\Gamma\left(1+s+n+\ell\right)}{\Gamma\left(1+n\right)\Gamma\left(1+n+\ell\right)} r^{\ell}\cos\left(\ell\theta\right) P_{n}^{(s,\ell)}\left(2r^{2}-1\right),
  % &\ell,n\geq0 \\
  f_{k,\ell}^{2D}(r,\theta)
  &= 2^{2s}\Gamma\left(1+s\right)^{2}\binom{s+k+\ell}{s}\binom{s+k}{s} r^{\ell}\cos\left(\ell\theta\right) P_{k}^{(s,\ell)}\left(2r^{2}-1\right),
  &\ell,k\geq0,
\end{align*}
then
%the solution of problem \eqref{eq:fractional_Poisson} is given, in polar coordinates, by
$
  u_{k,\ell}^{2D}(r,\theta) = r^{\ell}\cos\left(\ell\theta\right) P_{k}^{(s,\ell)}\left(2r^{2}-1\right)\left(1-r^{2}\right)^{s}_{+}.
$

We refer the reader to \cite{DydaKuznetsovEtAl2016_FractionalLaplaceOperatorMeijerGFunction} for details on how these 
%analytical 
solutions are determined.

We now construct analytic solutions to the parabolic fractional optimal control problem.
Let \(\psi, \phi \) be smooth functions on $(0,T)$ such that \(\psi(0)=1\) and \(\phi(T)=0\).
Let $f, g$ be smooth functions on $\Omega$ and \(u\) and \(v\) be the solutions to the fractional Poisson problem \eqref{eq:fractional_Poisson} with right-hand sides \(f\) and \(g\), respectively. Set
\begin{align*}
  \fsf\left(t,x\right) &= \psi'(t) u(x) + \psi(t) f(x) - \operatorname{proj}_{[\asf,\bsf]}\left(\phi(t) v(x)\right), \\
  \usf_{d}\left(t,x\right) &= \psi(t)u(x) + \mu\phi'(t)v(x) + \mu\phi(t)g(x),
\end{align*}
and $\usf_{0}\left(x\right) = u(x)$.
\EO{The exact solution to the optimal control problem is given by
$
  \bar \usf(t,x) = \psi(t)u(x),
$
$
\bar \psf(t,x) = -\mu\phi(t)v(x)
$
and  $\bar \zsf(t,x) = \proj_{[\asf,\bsf]}\left(\phi(t)v(x)\right)$.
Notice that $\bar \usf $, $\bar \psf$, and $\bar \zsf$ verify the regularity results of Theorems \ref{th:regularity_time} and \ref{thm:regularity_space}.}

\subsection{Implementation details}
In what follows, we employ the panel clustering approach described in~\cite{AinsworthGlusa2018_TowardsEfficientFiniteElement} to obtain a sparse approximation of the integral fractional Laplacian \(\left(-\Delta\right)^{s}\). For the minimization problem we use the BFGS algorithm~\cite{NocedalWright2006_NumericalOptimization2nd}.
The linear systems of equations arising from the fully discrete scheme from section \ref{sub:fd_control} are solved using conjugate gradient preconditioned by geometric multigrid.

The \(L^{2}\left(Q\right)\)-error of approximating the variable \(\bar \wsf\) by the discrete function \(\bar W^{\tau}_{\T}\) is approximated as follows:
\begin{align*}
  \|\bar \wsf - \bar W^{\tau}_{\T}\|_{L^{2}(Q)}^{2} &= \int_{0}^{T} \int_{\Omega} \left( \bar \wsf^{2} - 2 \bar \wsf \bar W^{\tau}_{\T} + \left(\bar W^{\tau}_{\T}\right)^{2} \right) \diff x \diff t
   \\
  &\approx \EO{\int_{0}^{T} \int_{\Omega} \bar \wsf^{2} \diff x \diff t+ \sum_{k=1}^{\K}
  \int_{\Omega} \bar W^{k}_{\T}\left(\bar W^{k}_{\T}-2\bar \wsf (t_k) \right).}
\end{align*}
Notice that the first term can be evaluated analytically.

%%%%%%%%%%%%%%%%%%%%%%%%%%%%%%%%%%%%%%%%%%%%%%%%%%%%%%%%%%%%%%%%%%%%%%%%%%%%%%%%%%%%%%
\subsection{Example in 1D}

We set \(\Omega= (0,1) \subset \mathbb{R}\), \(T=1\), \(\asf =-0.5\), \(\bsf =0.5\), \(\mu=0.1\), \(u=u_{0}^{1D}\), \(v=u_{0}^{1D}\), \(\psi(t)=\cos(t)\), and \(\phi(t)=\sin(T-t)\). 
The exact solution reads:
\[
  \bar \usf(t,x) = \cos(t)u_{0}^{1D}(x), 
  \qquad
  \bar \psf(t,x) = -\mu\sin(T-t)u_{0}^{1D}(x),
  \]
and
\begin{equation*}       
  \bar \zsf(t,x) = \proj_{[\asf,\bsf]}\left(\sin(T-t)u_{0}^{1D}(x)\right) =
                   \begin{cases}
     b & \text{if }\abs{x}<r_{o}(t), \\
     (1-x^{2})^{s} & \text{if }\abs{x}\geq r_{o}(t),
     \end{cases}
\end{equation*}
where
\begin{align*}
  r_{o}(t)&=
            \begin{cases}
              0 & \text{if }\sin(T-t)<b, \\
              \sqrt{1-\left(\frac{b}{\sin(T-t)}\right)^{1/s}} & \text{if }\sin(T-t)\geq b.
            \end{cases}
\end{align*}
We also set \(\tau=h_{\T}^{\gamma}\), where \(\gamma=\min\{1,s+1/2-\epsilon\}\) and $\epsilon >0$ is arbitrarily small.

%%In Figure~\ref{fig:solutions1D} we present the finite element solutions for the optimal state and control, on a suitable mesh, for \(s=0.7\). We observe that the constraint $\bsf = 0.5$ is indeed active for \(\abs{x}<r_{o}(t)\).
%%\begin{figure}
%%  \centering
%%  \includegraphics{{"PDF/parabolicOptControl-state-domain=interval-rhs=[0 0]-s=0.7-noRef=10-element=1-refinement=None-solver=MG-alpha=0.1-bounds=-0.5,0.5"}.pdf}
%%  \includegraphics{{"PDF/parabolicOptControl-control-domain=interval-rhs=[0 0]-s=0.7-noRef=10-element=1-refinement=None-solver=MG-alpha=0.1-bounds=-0.5,0.5"}.pdf}
%%  \caption{1D example: Finite element solutions for the optimal state \(\bar U^{\tau}_{\T}\) (left) and the optimal control \(\bar{Z}^{\tau}_{\T}\) (right) for \(s=0.7\).  We  notice  that  the  upper  bound  on  the  control  is active.}
%%  \label{fig:solutions1D}
%%\end{figure}

In Figure~\ref{fig:error1D} we display the experimental rates of convergence for the \(L^{2}(Q)\)-errors of the state and control variables.
We consider different values for the fractional order \(s\in\{0.1,0.2,\dots,0.9\}\).
\EO{We observe that the experimental rates of convergence for the error approximation of the control variable are in agreement with the error estimate of Theorem 
%\ref{thm:rate_state_final} and 
%\ref{thm:control_error}, i.e., 
%\(\|\bar \usf - \bar U^{\tau}_{\T}\|_{L^{2}(Q)} \sim h_{\T}^{\vartheta - 1/2 - \epsilon} 
%%+ \mathcal{T}^{\vartheta+1/2-\epsilon}
%\) with \(\vartheta=\min\{s,1/2-\epsilon\}\) and 
\(\| \bar{\zsf} - \bar{Z}^{\tau}_{\T}  \|_{L^2(Q)} \sim h_{\T}^{\gamma}\) with \(\gamma=\min\{1,s+1/2-\epsilon\}\).}

\begin{figure}
  \centering
  \includegraphics{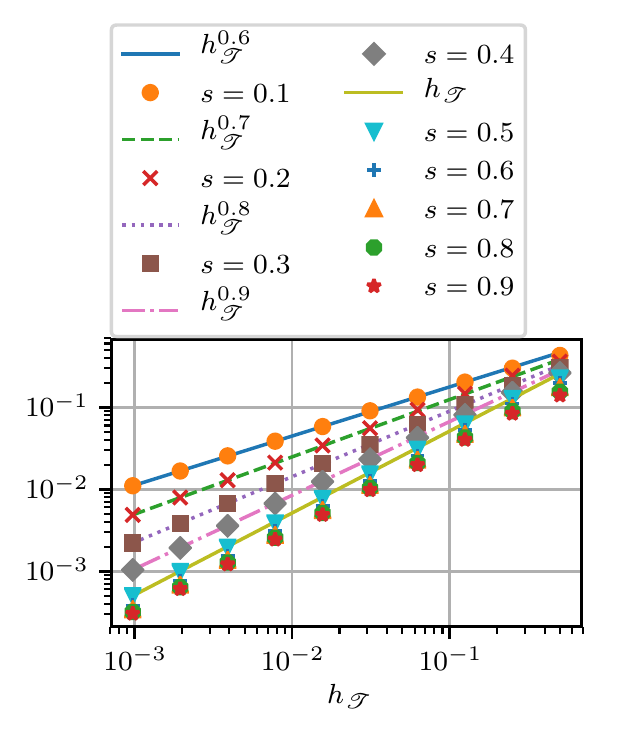}
  \includegraphics{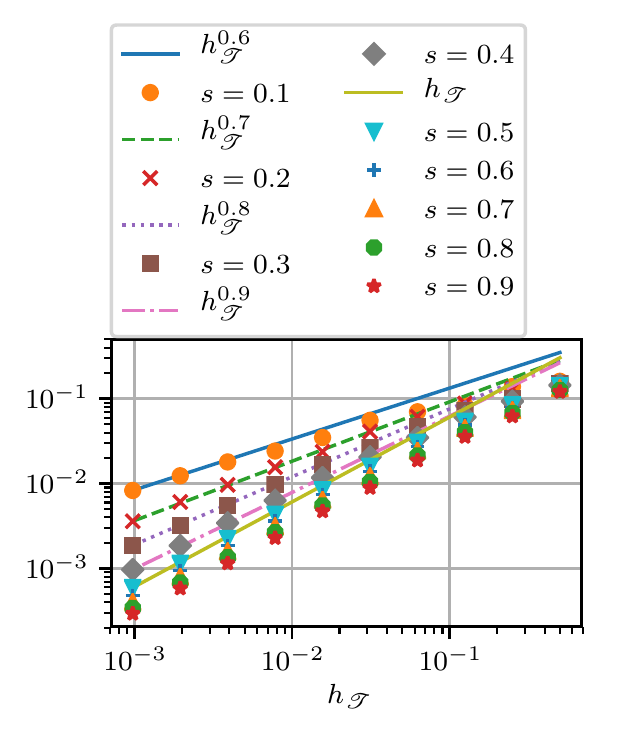}
  \caption{1D example: Experimental rates of convergence for \(\|\bar \usf - \bar U^{\tau}_{\T}\|_{L^{2}(Q)}\) (left) and \(\| \bar{\zsf} - \bar{Z}^{\tau}_{\T}  \|_{L^2(Q)}\) (right) for $\Omega = (0,1)$ and \(s\in\{0.1, 0.2, \dots, 0.9\}\).}
  \label{fig:error1D}
\end{figure}

%%%%%%%%%%%%%%%%%%%%%%%%%%%%%%%%%%%%%%%%%%%%%%%%%%%%%%%%%%%%%%%%%%%%%%%%%%%%%%%%%%%%%%
\subsection{Examples in 2D}

\EO{We set \(\Omega=B(0,1)\subset \mathbb{R}^{2}\), \(T=1\), \(\asf =-0.5\), \(\bsf=0.5\) and \(\mu=0.1\).
We consider two problems:
\begin{itemize}
\item[(I)]
  Set \(u=u_{0,1}^{2D}\), \(v=u_{0,0}^{2D}\), \(\psi(t)=\cos(t)\), and \(\phi(t)=\sin(T-t)\).
  The exact solution to the fractional optimal control problem is then given by
  \begin{equation*}
    \bar \usf(t,x) = \cos(t)u_{0,1}^{2D}(x),
    \quad
    \bar \psf(t,x) = -\mu\sin(T-t)u_{0,0}^{2D}(x),
  \end{equation*}
  and $\bar \zsf(t,x) = \proj_{[\asf,\bsf]}\left(\sin(T-t)u_{0,0}^{2D}(x)\right)$.

\item[(II)]
  Let \(\fsf(t,x)=\cos(t)\), \(\usf_{d}\left(t,x\right)=\cos(t)(1-\abs{x}^{2})\) and \(\usf_{0}\left(x\right) = 1-\abs{x}^{2}\).
  No analytic expressions for \(\bar \usf\), \(\bar \psf\) or \(\bar \zsf\) are available.

\end{itemize}
}

%%%%%%%%%%%%%%%%%%%%%%%%%%%%%%%%%%%%%%%%%%%%%%%%%%%%%%%%%%%%%%%%%%%%%%%%%%%%%%%%%%%%%%x
\subsubsection{Quasi-uniform meshes}

\EO{We solve the fully discrete scheme on quasi-uniform meshes with mesh sizes \(h_{\T}\) and time steps of size \(\tau=h_{\T}^{\gamma}\), where \(\gamma=\min\{1,s+1/2-\epsilon\}\) and $\epsilon > 0$ is arbitrarily small.

In Figure~\ref{fig:error2D} we present, for (I) and $s = 0.25$ and $s=0.75$, the experimental rates of convergence for the $L^{2}(Q)$-errors of the state and control variable as well as the \(L^{2}((0,T),H^{s}(\Omega))\)-error of the state variable.
Moreover, we also show the $L^{2}(Q)$-errors of the state and control for (II), computed with respect to a very fine solution.
%We observe the predicted rates of convergence for the $L^{2}(Q)$-error of the control variable.
We observe that the experimental rates of convergence for the error approximation of the control variable are in agreement with the error estimate \eqref{eq:final_error_estimate_control} of Theorem \ref{thm:control_error}.
The slightly faster convergence for (II) is explained by the fact that fine solutions are used as reference to compute errors.}
%that the  predicted by Theorems \ref{thm:rate_state_final} and \ref{LM:control_error} are obtained.

%we display \EO{the experimental rates of convergence for the \(L^{2}(Q)\)-errors of state and control variables. We consider different values for the fractional order} \(s\in\{0.1,0.2,\dots,0.9\}\) and different mesh sizes \(h_{\T}\). We observe the predicted rates of convergence.

\begin{figure}
  \centering
  \includegraphics{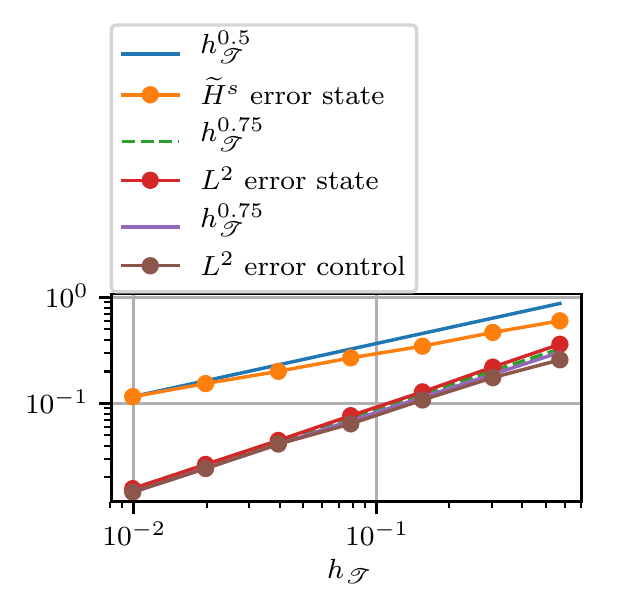}
  \includegraphics{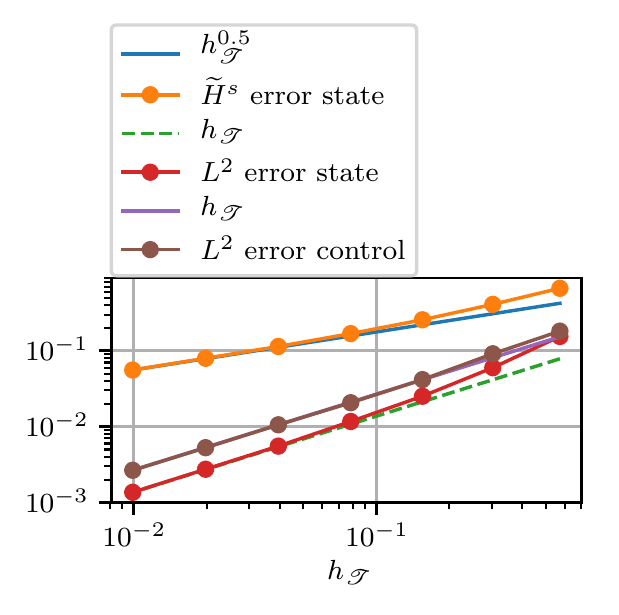}
  \includegraphics{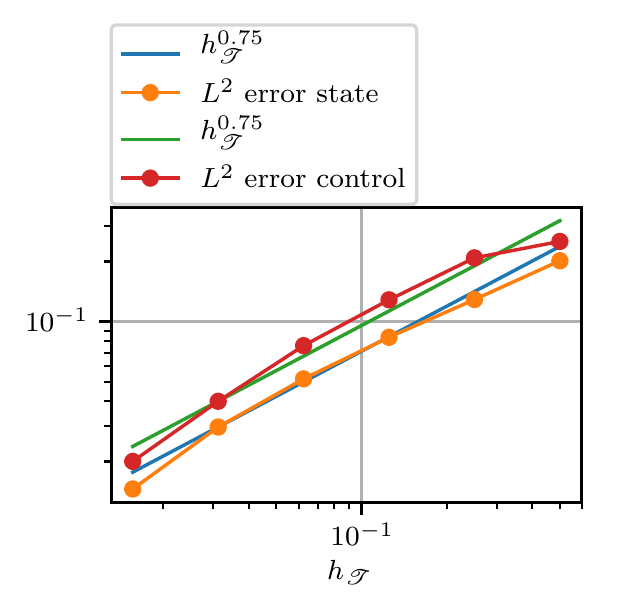}
  \includegraphics{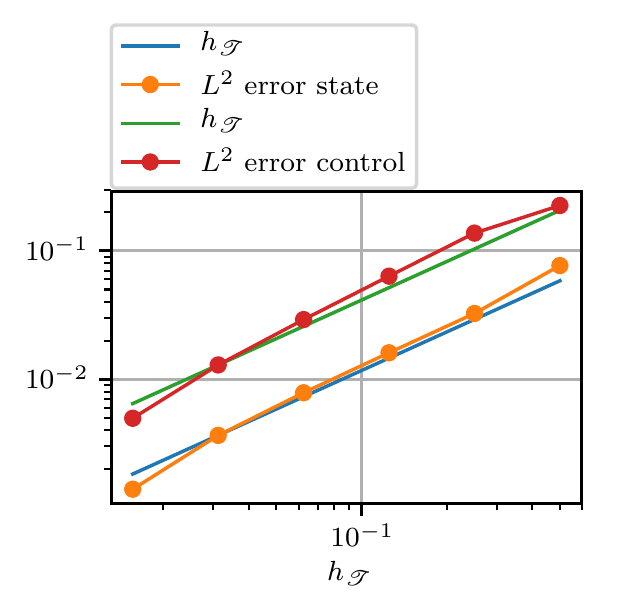}
  \caption{2D example: Experimental rates of convergence for \(\|\bar \usf - \bar U^{\tau}_{\T}\|_{L^{2}((0,T),H^{s}\left(\Omega\right))}\), \(\|\bar \usf - \bar U^{\tau}_{\T}\|_{L^{2}(Q)}\), and \(\| \bar{\zsf} - \bar{Z}^{\tau}_{\T}  \|_{L^2(Q)}\) for \(s=0.25\) (left) and \(s=0.75\) (right) on quasi-uniform meshes for problems (I) (\emph{top}) and (II) (\emph{bottom}).}
  \label{fig:error2D}
\end{figure}

%%%%%%%%%%%%%%%%%%%%%%%%%%%%%%%%%%%%%%%%%%%%%%%%%%%%%%%%%%%%%%%%%%%%%%%%%%%%%%%%%%%%%%
\subsubsection{Graded meshes}

In this section we explore the computational performance of the devised fully discrete scheme on the basis of finite elements spaces over graded meshes on $\Omega$.
We motivate and describe such a graded finite element setting in what follows.
When $s \in (1/2,1)$ and $n=2$, the singular behavior of the solution to the elliptic fractional Poisson problem \eqref{eq:fractional_Poisson} can be compensated by using a priori adapted meshes; see~\cite{MR3620141}.
These graded meshes, which allow for an improvement on the a priori error estimate obtained on quasiuniform meshes, are constructed as follows.
In addition to shape regularity, we assume that the meshes $\T$ have the following property:
Given a mesh parameter $h>0$ and $\kappa \in [1,2]$ every element $T \in \T$ satisfies
\begin{equation*}
h_T \approx C(\sigma) h^{\kappa} \textrm{ if } T \cap \partial \Omega \neq \emptyset,
\quad
h_T \approx C(\sigma) h \mathrm{dist}(T,\partial \Omega)^{(\kappa-1)/\kappa} \textrm{ if } T \cap \partial \Omega = \emptyset,
\end{equation*}
where $C(\sigma)$ depends only on the shape regularity constant $\sigma$ of the mesh $\T$.
$\kappa$ relates the mesh parameter $h$ to the number of degrees of freedom, $N$, as follows:
$
N \approx h^{-2}\textrm{ if } \kappa \in (1,2), \quad
N \approx h^{-2}|\log h|\textrm{ if } \kappa = 2.
$
The optimal choice is $\kappa =2$.

In Figure~\ref{fig:error2Dgraded} we present the experimental rates of convergence for  \(\|\bar \usf - \bar U^{\tau}_{\T}\|_{L^{2}(Q)}\), \(\|\bar \usf - \bar U^{\tau}_{\T}\|_{L^{2}((0,T),H^{s}\left(\Omega\right))}\) and \(\| \bar{\zsf} - \bar{Z}^{\tau}_{\T}  \|_{L^2(Q)}\) obtained by using graded meshes on $\Omega$ with grading parameter \(\kappa=2\) for problem (I).
%
%the errors obtained for the optimal control problem using a graded mesh with grading parameter \(\kappa=2\).
We observe improved rates of convergence for the error approximation of the state variable in both \(L^{2}(Q)\)- and \(L^{2}((0,T),H^{s}\left(\Omega\right))\)-norms.
We note that this setting is not covered by the analysis developed in the previous sections; the main missing ingredient being regularity estimates for the solution of \eqref{eq:fractional_heat} over bounded and Lipschitz domains $\Omega \times (0,T)$.

\begin{figure}
  \centering
  \includegraphics{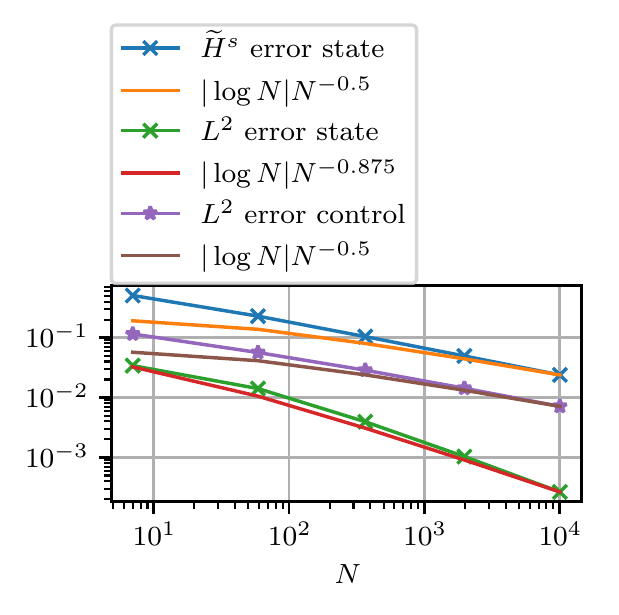}
  \caption{2D example on granded meshes: Experimental rates of convergence for  \(\|\bar \usf - \bar U^{\tau}_{\T}\|_{L^{2}(Q)}\), \(\|\bar \usf - \bar U^{\tau}_{\T}\|_{L^{2}(0,T,H^{s}\left(\Omega\right))}\), and \(\| \bar{\zsf} - \bar{Z}^{\tau}_{\T}  \|_{L^2(Q)}\) for \(s=0.75\) on problem (I).}
  \label{fig:error2Dgraded}
\end{figure}

\section{Conclusions}

We have analyzed a control-constrained linear-quadratic optimal control problem for the fractional heat equation and derived existence and uniqueness results, first order optimality conditions, and regularity estimates for the optimal variables. We have proposed a fully discrete scheme to discretize the state equation equation that relies on an implicit finite difference discretization in time combined with a piecewise linear finite element discretization in space. We have derived stability results and a priori error estimate in $L^2(0,T;L^2(\Omega))$. Furthermore, we have proposed a fully discrete scheme for the optimal control problem that discretizes the control variable with piecewise constant functions, and derived a priori error estimates for it. Finally, we have illustrated the theory with one-- and two--dimensional numerical experiments.

\section{Acknowledgments}
We would like to thank the anonymous referees for several comments and suggestions that led to better results and an improved presentation.
E. Ot\'arola would also like to thank G. Grubb for insightful discussions on regularity properties for fractional heat equations.
E. Ot\'arola was supported by CONICYT through FONDECYT project 11180193.
C. Glusa was supported by Sandia National Laboratories (SNL) and the Laboratory Directed Research and Development program at SNL.
SNL is a multimission laboratory managed and operated by National Technology and Engineering Solutions of Sandia, LLC., a wholly owned subsidiary of Honeywell International, Inc., for the U.S. Department of Energy’s National Nuclear Security Administration contract number DE-NA-0003525.
This paper describes objective technical results and analysis.
Any subjective views or opinions that might be expressed in the paper do not necessarily represent the views of the U.S. Department of Energy or the United States Government.
SAND Number: SAND2020-6344 O

\bibliographystyle{plain}
\bibliography{biblio}

\end{document}